\documentclass{conm-p-l}
\usepackage{amssymb}

\usepackage{amsfonts}

\newtheorem{theorem}{Theorem}[section]
\newtheorem{lemma}[theorem]{Lemma}
\theoremstyle{definition}
\newtheorem{definition}[theorem]{Definition}
\newtheorem{example}[theorem]{Example}

\theoremstyle{remark}
\newtheorem{remark}[theorem]{Remark}
\numberwithin{equation}{section}
\theoremstyle{plain}
\newtheorem{acknowledgement}{Acknowledgement}

\newtheorem{corollary}{Corollary}

\newtheorem{proposition}{Proposition}

\copyrightinfo{2001}{enter name of copyright holder}

\input{tcilatex}

\begin{document}
\title[$\ast $-homogeneous ideals]{On $\ast $-homogeneous ideals}
\author{Muhammad Zafrullah}
\address{Department of Mathematics, Idaho State University,\\
Pocatello, Idaho 83209 USA}
\email{mzafrullah@usa.net}
\urladdr{http://www.lohar.com}
\subjclass[2010]{13A15; Secondary 13G05; 06F20}
\date{July 27, 2001}
\keywords{star operation, Monoid, pre-Riesz Monoid, $\ast $-homogeneous $%
\ast $-potent domain}
\dedicatory{Dedicated to the memory of True Friendship}

\begin{abstract}
Let $\ast $ be a star operation of finite character. Call a $\ast $-ideal $I$
of finite type a $\ast $-homogeneous ideal if $I$ is contained in a unique
maximal $\ast $-ideal $M=M(I).$ A maximal $\ast $-ideal that contains a $%
\ast $-homogeneous ideal is called potent and the same name bears a domain
all of whose maximal $\ast $-ideals are potent. One among the various aims
of this article is to indicate what makes a $\ast $-ideal of finite type a $%
\ast $-homogeneous ideal, where and how we can find one, what they can do
and how this notion came to be. We also prove some results of current
interest in ring theory using some ideas from this author's joint work in 
\cite{LYZ 2014} on partially ordered monoids. We characterize when a
commutative Riesz monoid generates a Riesz group
\end{abstract}

\maketitle

\section{Introduction}

Let $\ast $ be a finite character star operation defined on an integral
domain $D$ throughout. (A working introduction to the star operations, and
the reason for using them, will follow.) Call a nonzero $\ast $-ideal of
finite type a $\ast $-homogeneous ideal, if $I$ is contained in a unique
maximal $\ast $-ideal. According to proposition $1$ of \cite{AZ 2019},
associated with each $\ast $-homogeneous ideal $I$ is a unique $\ast $%
-maximal ideal $M(I)=\{x\in D|$ $(x,I)^{\ast }\neq D\}.$ The notion of a $%
\ast $-homogeneous ideal has figured prominently in describing unique
factorization of ideals and elements in \cite{AZ 2019} and it seems
important to indicate some other properties and uses of this notion and
notions related to it. Call a $\ast $-maximal ideal $M$ $\ast $-potent if $M$
contains a $\ast $-homogeneous ideal and call a domain $D$ $\ast $-potent if
each of the $\ast $-maximal ideals of $D$ is $\ast $-potent. The aim of this
article is to study some properties of $\ast $-homogeneous ideals and of $%
\ast $-potent domains. We show for instance that while in a $\ast $-potent
domain every proper $\ast $-ideal of finite type is contained in a $\ast $%
-homogeneous ideal, the converse may not be true. We shall also indicate how
these concepts can be put to use. Before we elaborate on that, it seems
pertinent to give an idea of our main tool, the star operations. Indeed, the
rest of what we plan to prove will be included in the plan of the paper
after the introduction to star operations.

\subsection{Introduction to star operations}

Let $D$ be an integral domain with quotient field $K$ $,$ throughout. Let $%
F(D)$ be the set of nonzero fractional ideals of $D,$ and let $f(D)=\{A\in
F(D)|A$ is finitely generated\}. A star operation $\ast $ on $D$ is a
closure operation on $F(D)$ that satisfies $D^{\ast }=D$ and $(xA)^{\ast
}=xA^{\ast }$ for $A\in F(D)$ and $x\in K=K\backslash \{0\}$. With $\ast $
we can associate a new star-operation $\ast _{s}$ given by $A\mapsto A^{\ast
_{s}}=\cup \{B^{\ast }|B\subseteq A,B\in f(D)\}$ for each $A\in F(D).$ We
say that $\ast $ has finite character if $\ast =\ast _{s}$. Three important
star-operations are the $d$-operation $A\mapsto A_{d}=A$, the $v$-operation $%
A\mapsto A_{v}=(A^{-1})^{-1}$ $=\cap \{Dx|Dx\supseteq A,x\in K\}$ where $%
A^{-1}=\{x\in K:xA\subseteq D\}$ and the $t$-operation $t=v_{s}.$ Here $d$
and $t$ have finite character. A fractional ideal $A$ is a $\ast $-ideal if $%
A=A^{\ast }$and a $\ast $-ideal $A$ is of finite type if $A=B^{\ast }$ for
some $B\in f(D).$ If $\ast $ has finite character and $A^{\ast }$ is of
finite type, then $A^{\ast }=B^{\ast }$ for some $B\in f(D),$ $B\subseteq A.$
A fractional ideal $A\in F(D)$ is $\ast $-invertible if there exists a $B\in
F(D)$ with $(AB)^{\ast }=D$; in this case we can take $B=A^{-1}$. For any $%
\ast $-invertible $A\in F(D),A^{\ast }=A_{v}$. If $\ast $ has finite
character and $A$ is $\ast $-invertible, then $A^{\ast }$ is a finite type $%
\ast $-ideal and $A^{\ast }=A_{t}$ . Given two fractional ideals $A,B\in
F(D),~(AB)^{\ast }$ denotes their $\ast $-product. Note that $(AB)^{\ast
}=(A^{\ast }B)^{\ast }=(A^{\ast }B^{\ast })^{\ast }$. Given two star
operations $\ast _{1}$ and $\ast _{2}$ on $D$, we write $\ast _{1}\leq \ast
_{2}$ if $A^{\ast _{1}}\subseteq A^{\ast _{2}}$ for all $A\in F(D)$. So $%
\ast _{1}\leq \ast _{2}\Leftrightarrow (A^{\ast _{1}})^{\ast _{2}}=A^{\ast
_{2}}\Leftrightarrow (A^{\ast _{2}})^{\ast _{1}}=A^{\ast _{2}}$ for all $%
A\in F(D).$

Indeed for any finite character star-operation $\ast $ on $D$ we have $d\leq
\ast \leq t$ . For a quick introduction to star-operations, the reader is
referred to \cite[Sections 32, 34]{Gil 1972} or \cite{Zaf 2000}, for a quick
review. For a more detailed treatment see Jaffard \cite{Jaf 1962}. A keenly
interested reader may also look up \cite{H-K 1998}. These days star
operations are being used to define analogues of various concepts. The trick
is to take a concept, e.g., a PID and look for what the concept would be if
we require that for every nonzero ideal $I,$ $I^{\ast }$ is principal and
voila! You have several concepts parallel to that of a PID. Of these $t$-PID
turns out to be a UFD. Similarly a $v$-PID is a completely integrally closed
GCD domain of a certain kind. A $t$-Dedekind domain, on the other hand is a
Krull domain and a $v$-Dedekind domain is a domain with the property that
for each nonzero ideal $A$ we have $A_{v}$ invertible. So when we prove a
result about a general star operation $\ast $ the result gets proved for all
the different operations, $d,t,v$ etc. Apart from the above, any terminology
that is not mentioned above will be introduced at the point of entry of the
concept.

Suppose that $\ast $ is a finite character star-operation on $D$. Then a
proper $\ast $-ideal is contained in a maximal $\ast $-ideal and a maximal $%
\ast $-ideal is prime. We denote the set of maximal $\ast $-ideals of $D$ by 
$\ast $-$Max(D)$. We have $D=\cap D_{P}$ where $P$ ranges over $\ast $-$%
Max(D).$From this point on we shall use $\ast $ to denote a finite type star
operation. Call $D$ of finite $\ast $-character if for each nonzero non unit 
$x$ of $D,$ $x$ belongs to at most a finite number of maximal $\ast $%
-ideals. Apart from the introduction there are three sections in this paper.
In section \ref{S1} we talk about $\ast $-homogeneous ideals, and $\ast $%
-potent domains. We characterize $\ast $-potent domains in this section,
show that if $D$ is of finite $\ast $-character then $D$ must be potent,
examine an error in a paper of the author, \cite{DZ 2010}, in characterizing
domains of finite $\ast $-character and characterize domains of finite $\ast 
$-character and give a new proof. In section \ref{S2}, we show how creating
a suitable definition of a $\ast $-homogeneous ideal will create theory of
unique factorization of ideals. Calling an element $r\in D$ $\ast $-f-rigid (%
$\ast $-factorial rigid) if $rD$ is a $\ast $-homogeneous ideal such that
every proper $\ast $-homogeneous ideal containing $r$ is principal we call a 
$\ast $-potent maximal $\ast $-ideal $M$ (resp., domain $D$) $\ast $%
-f-potent if $M$ (resp., every maximal $\ast $-ideal of $D$) contains a $%
\ast $-f-rigid element and show that over a $\ast $-f-potent domain a
primitive polynomial $f$ is super primitive i.e. if $A_{f},$ the content of $%
f,$ is such that the generators of $f$ have no non unit common factor then $%
(A_{f})_{v}=D$ and indicate how to construct atomless non-pre-Schreier
domain. In this section we offer a seamless patch to remove an error in the
proof of result in a paper by Kang \cite{Kan 1989} and show that $D$ is a $t$%
-superpotent if and only if $D[X]_{S}$ is $t$-f-potent, where $X$ is an
indeterminate and $S=\{f\in \lbrack X]|(A_{f})_{v}=D\}.$ We also show, by
way of constructing more examples, in this section that if $L$ is an
extension of $K$ the quotient field of $D$ and $X$ an indeterminate over $D$
then $D$ $t$-f-potent if and only if $D+XL[X]$ is. Finally in section \ref%
{S3} we define a pre-Riesz monoid as a p.o. monoid $M$ if for any $%
x_{1},x_{2},...,x_{n}\in M\backslash \{0\}$ $glb(x_{1},x_{2},...,x_{n})=0$
or there is $r\in M$ with $0<r\leq x_{1},x_{2},...,x_{n}$ and indicate that
the monoid of $\ast $-ideals of finite type is a pre-Riesz monoid and, of
course we indicate how to use this information.

\section{\label{S1}$\ast $-potent domains and $\ast $-homogeneous ideals}

Work on this paper started in earnest with the somewhat simple observation
that if $D$ is $\ast $-potent then every nonzero non unit $x\in D$ is
contained in some $\ast $-homogeneous ideal. The proof goes as follows:
Because $x$ is a nonzero non unit, $x$ must be contained in some maximal $%
\ast $-ideal $M.$ Now as $D$ is $\ast $-potent $M=M(I)$ for some $\ast $%
-homogeneous ideal $I.$ Consider $J=(I,x)^{\ast }$ and note that $%
(I,x)^{\ast }\neq D$ because $x\in M$ and $(I,x)^{\ast }$ is contained in a
unique maximal $\ast $-ideal and this makes $J$ a $\ast $-homogeneous ideal.

This leads to the question: If $D$ is a domain with a finite character star
operation $\ast $ defined on it such that every nonzero non unit $x$ of $D$
is contained in some $\ast $-homogeneous ideal $I$ of $D,$ must $D$ be $\ast 
$-potent?

This question came up in a different guise as: when is a certain type of
domain $\ast _{s}$-potent for a general star operation $\ast $ in \cite{YZ
2011} and sort of settled in a tentative fashion in Proposition 5.12 of \cite%
{YZ 2011} saying, in the general terms being used here, that: Suppose that $%
D $ is a domain with a finite character $\ast $-operation defined on it.
Then $D$ is $\ast $-potent provided (1) every nonzero non unit $x$ of $D$ is
contained in some $\ast $-homogeneous ideal $I$ of $D$ and (2) for $M,$ $%
M_{\alpha }\in $ $\ast $-$max(D)$, $M\subseteq \cup M_{\alpha }$ implies $%
M=M_{\alpha }$ for some $\alpha .$

The proof could be something like: By (1) for every nonzero non unit $x$
there is a $\ast $-homogeneous ideal $I_{x}$ containing $x$ and so $x\in
M(I_{x})$. So $M\subseteq \cup M(I_{x})$ and by (2) $M$ must be equal to $%
M(I_{x})$ for some $x.$

Thus we have the following statement.

\begin{theorem}
\label{Theorem 1}. Let $\ast $ be a finite character star operation defined
on $D.$ Then $D$ is $\ast $-potent if $D$ satisfies the following: (1)every
nonzero non unit $x$ of $D$ is contained in some $\ast $-homogeneous ideal $%
I $ of $D$ and (2) For $M,$ $M_{\alpha }\in $ $\ast $-$max(D)$, $M\subseteq
\cup M_{\alpha }$ implies $M=M_{\alpha }$ for some $\alpha .$
\end{theorem}

Condition (2) in the statement of Theorem \ref{Theorem 1} has had to face a
lot of doubt from me, in that, is it really necessary or perhaps can it be
relaxed a little?

The following example shows that condition (2) or some form of it is here to
stay.

It is well known that the ring $\mathcal{E}$ of entire functions is a Bezout
domain \cite[Exercise 18, p 147]{Gil 1972}. It is easy to check that a
principal prime in a Bezout domain is maximal. Now we know that a zero of an
entire function determines a principal prime in $\mathcal{E}$ and that the
set of zeros of a nontrivial entire function is discrete, including
multiplicities, the multiplicity of a zero of an entire function is a
positive integer \cite[Theorem 6]{Hel 1940}. Thus each nonzero non unit of $%
\mathcal{E}$ is expressible as a countable product of finite powers of
distinct principal primes. For the identity star operation $d,$ certainly
defined on $\mathcal{E},$ only an ideal $I$ generated by a power of a
principal prime can be $d$-homogeneous. For if $I$ is $d$-homogeneous, then $%
I=(x_{1},...,x_{n})^{d}=x\mathcal{E}$ a principal ideal and hence a
countable product of distinct primes. Now $I$ cannot be in a unique non
principal prime for then $I$ would have to be a countably infinite product
of principal primes and so in infinitely many principal prime ideals, which
are maximal. So $I$ can only belong to a unique principal prime and has to
be a finite prime power. To see that $\mathcal{E}$ falls foul of Theorem \ref%
{Theorem 1}, let's put $S=\{p|p$ a prime element in $\mathcal{E}\}.$ Then
for each non principal prime $P$ of $\mathcal{E}$ we have $P\subseteq \cup
_{p\in S}p\mathcal{E}$ because each element of $P$ is divisible by some
member(s) of $S.$ (I have corresponded with Prof. Evan Houston about the
above material and I gratefully acknowledge that.)

Once we know more about $\ast $-homogeneous ideals we would know that rings
do not behave in the same manner as groups do. To get an idea of how groups
behave and what is the connection the reader may look up \cite{YZ 2011}.
Briefly, the notion of a $\ast $-homogeneous ideal arose from the notion of
a basic element of a lattice ordered group $G$ (defined as $b>0$ in $G$ such
that $(0,b]$ is a chain). A basis of $G,$ if it exists, is a maximal set of
mutually disjoint strictly positive basic elements of $G$. According to \cite%
{Con 1961} a l.o. group $G$ has a basis if and only if every strictly
positive element of $G$ exceeds a basic element. So if we were to take $D$
being potent as having a basis (every proper $\ast $-ideal of finite type
being contained in a $\ast $-homogeneous ideal) then every proper $\ast $%
-ideal of finite type being contained in a $\ast $-homogeneous ideal does
not imply that $D$ is potent.

We next tackle the question of where $\ast $-homogeneous ideals can be
found. Call $D$ of finite $\ast $-character if every nonzero non unit of $D$
is contained in at most a finite number of maximal $\ast $-ideals. Again, a
domain of finite $\ast $-character could be a domain of finite character
(every nonzero non unit belongs to at most a finite number of maximal
ideals) such as an h-local domain or a semilocal domain or a PID or a domain
of finite $t$-character such as a Krull domain.

\begin{proposition}
\label{Proposition X1} A domain $D$ of finite $\ast $-character is $\ast $%
-potent.
\end{proposition}

\begin{proof}
Let $M$ be a maximal $\ast $-ideal of $D$ and let $x$ be a nonzero element
of $M.$ If $x$ belongs to no other maximal $\ast $-ideal then $xD$ is $\ast $%
-homogeneous and $M$ is potent. So let us assume that $M,$ $%
M_{1},M_{2},...,M_{n}$ is the set of all maximal $\ast $-ideals containing $%
x.$ Now consider the ideal $A=(x,x_{1},...,x_{n})$ where $x_{i}\in
M\backslash M_{i}$ for $i=1,...,n.$ Obviously $A\subseteq M$ but $%
A\nsubseteq M_{i}$ because of $x_{i}.$ Note that $A$ cannot be contained in
any maximal $\ast $-ideal other than $M,$ for if $N$ were any maximal $\ast $%
-ideal containing $A$ then $N$ would belong to $\{M,M_{1},M_{2},...,M_{n}\}$
because of $x.$ And $N$ cannot be any of the $M_{i}.$ Thus $A^{\ast }$ is a $%
\ast $-homogeneous ideal contained in $M$ and $M$ is potent. Since $M$ was
arbitrary we have the conclusion.
\end{proof}

The above proof is essentially taken from the proof for part (2) of Theorem
1.1 of \cite{ACZ 2013}.

Now how do we get a domain of finite $\ast $-character? The answer is
somewhat longish and interesting. Bazzoni conjectured in \cite{Baz 2000} and 
\cite{Baz 2001} that a Prufer domain $D$ is of finite character if every
locally principal ideal of $D$ is invertible. \cite{HMMT 2008} were the
first to verify the conjecture using partially ordered groups. Almost
simultaneously \cite{H-K 2009} proved the conjecture for $r$-Prufer monoids,
using Clifford semigroups of ideals and soon after I chimed in with a very
short paper \cite{Zaf 2010}. The ring-theoretic techniques used in this
paper not only verified the Bazzoni conjecture but also helped prove
Bazzoni-like statements for other, suitable, domains that were not
necessarily PVMDs. (Recall that $D$ is a PVMD if every $t$-ideal $A$ of
finite type of $D$ is $t$-invertible i.e. $(AA^{-1})_{t}=D$.) In the course
of verification of the conjecture I mentioned a result due to Griffin from 
\cite{Gri 1967} that says:

\begin{theorem}
\label{Theorem X2} A PVMD\ $D$ is of finite $t$-character if and only if
each $t$-invertible $t$-ideal of $D$ is contained in at most a finite number
of mutually $t$-comaximal $t$-invertible $t$-ideals of $D$.
\end{theorem}

As indicated in the introduction of \cite{Zaf 2010} the set of $t$%
-invertible $t$-ideals of a PVMD is a lattice ordered group under $t$%
-multiplication and the order defined by reverse containment of the ideals
involved and that the above result for PVMDs came from the use of Conrad's
F-condition. Stated for lattice ordered groups Conrad's F-condition says:
Every strictly positive element exceeds at most a finite number of mutually
disjoint elements. This and Theorem \ref{Theorem X2}, eventually led the
authors of \cite{DZ 2010}, to the following statement.

\begin{theorem}
\label{Theorem X3}(cf. Theorem 1 of \cite{DZ 2010}) Let $D$ be an integral
domain, $\ast $ a finite character star operation on $D$ and let $\Gamma $
be a set of proper, nonzero, $\ast $-ideals of finite type of $D$ such that
every proper nonzero $\ast $-finite $\ast $-ideal of $D$ is contained in
some member of $\Gamma $ . Let $I$ be a nonzero finitely generated ideal of $%
D$ with $I^{\ast }\neq D$. Then $I$ is contained in an infinite number of
maximal $\ast $-ideals if and only if there exists an infinite family of
mutually $\ast $-comaximal ideals in $\Gamma $ containing $I$.
\end{theorem}

This theorem was a coup, it sort of catapulted the consideration of
finiteness of character from Prufer-like domains to consideration of
finiteness of $\ast $-character in general domains. But alas, there was an
error in the proof. There was no reason for the error as I had used the
technique, Conrad's F-condition, involved in the proof of Theorem \ref%
{Theorem X3} at other places such as \cite{DLMZ 2001}, \cite{MRZ 2008} and,
later, \cite{DZ 2011} but there it was. I realized the error while working
on a paper on p.o. groups, that I eventually published with Y.C. Yang as 
\cite{YZ 2011}. I wrote to my coauthor of \cite{DZ 2010}, proposing a
corrigendum. But for one reason or another the corrigendum never got off the
ground. Fortunately Chang and Hamdi have recently published \cite{CH 2019}
including Theorem 1 of \cite{DZ 2010} as Lemma 2.3 with proof exactly the
way I would have liked after the corrigendum was used.

Perhaps as a kind gesture those authors have not pointed out the error in
the proof of \cite[Theorem 1]{DZ 2010}, but a careless use of Zorn's Lemma
must be pointed out so that others do not fall in a similar pit. Now going
over the whole thing anew might be painful, so I reproduce below the
proposed brief corrigendum and point out any other s made that I could not
see at that time.

\begin{quotation}
\textquotedblleft There is some confusion in lines 8-15 of the proof of
Theorem 1. In the following we offer a fix to clear the confusion and give a
rationale for the fix.

The fix: Read the proof from the sentence that starts from line 8 as
follows: Let $S$ be the family of sets of mutually $\ast $-comaximal
homogeneous members of $\Gamma $ containing $I$. Then $S$ is non empty by $%
(\sharp \sharp ).$ Obviously $S$ is partially ordered under inclusion. Let $%
A_{n_{1}}\subset A_{n_{2}}\subset ...\subset A_{n_{r}}\subset ...$ be an
ascending chain of sets in $S$. Consider $T=\cup A_{n_{r}}.$ We claim that
the members of $T$ are mutually $\ast $-comaximal. For take $x,y\in T,$ then 
$x,y\in A_{n_{i}},$ for some $i,$ and hence are $\ast $-comaximal. Having
established this we note that by $(\sharp ),$ $T$ must be finite and hence
must be equal to one of the $A_{n_{j}}.$ Thus by Zorn's Lemma, $S$ must have
a maximal element $U=\{V_{1},V_{2},...,V_{n}\}.$ Disregard the next two
sentences and read on from: Next let $M_{i}$ be the maximal $\ast $-ideal....

Rationale for the Fix: Using sets of mutually $\ast $-comaximal elements
would entail some unwanted maximal elements as the following example shows:
Let $x=2^{2}5^{2}$ in $Z$ the ring of integers. Then $\mathcal{S}%
=\{\{(2^{2}5^{2})\},\{(25^{2})\},\{(2^{2}5)\}\{(2^{2})\},$ $%
\{(5^{2})\},\{(2^{2}),(5^{2})\},$ $\{(2)\},\{(5)\},$ $\{(2),(5^{2})\},%
\{(2^{2}),(5)\},$ $\{(2),(5)\}\}.$ In this case, while $\mathcal{S}$
includes legitimate maximal elements: $\{(2^{2}),(5^{2})\},$ $%
\{(2),(5^{2})\},$ $\{(2^{2}),(5)\},\{(2),(5)\}$ it also includes $%
\{(2^{2}5^{2})\},\{(25^{2})\},\{(2^{2}5)\}$ which fit the definition of
maximal elements. The reason why the fix should work is that given any set $%
T=\{A_{1},A_{2},...,A_{m}\}$ of mutually $\ast $-comaximal $\ast $-finite
ideals, by $(\sharp \sharp )$ there is a set of mutually $\ast $-comaximal
homogeneous $\ast $-finite ideals $\{H_{1},H_{2},...,H_{n}\}$ in $\Gamma ,$
where $n\geq m$ such that each $H_{j}$ contains some $A_{i}.$ Also as a
homogeneous ideal cannot be contained in two disjoint ideals we do not face
the above indicated problem and Zorn's Lemma gives the required maximal
elements.\textquotedblright\ 

(To be sure that the above "proposal" was not created after seeing the Chang
Hamdi paper check the image of the E-mail sent to Prof. Dumitrescu and a pdf
version of the corrigendum here \cite{Zaf1 2019}, at the end of that
doument.)
\end{quotation}

The other error was essentially confusing the size of a set with the set, on
my part. I must admit that my coauthor told me to say, after finding that
there was at least one homogeneous ideal containing a given $\ast $-ideal $A$
of finite type, that one can find a largest set of mutually $\ast $%
-comaximal homogeneous ideals containing $A.$ But I just don't care about
doing that unless the conclusion is very simple.

It's only fitting that I end this saga with a more satisfying statement
and/or proof of \cite[Theorem 1]{DZ 2010}. Lurking behind the fa\c{c}ade of
the set $\Gamma $ and the other conditions were the following definitions
and statements. Call a $\ast $-ideal $I$ of finite type ($\ast $-)
homogeneous, as we have already done, if $I$ is contained in a unique
maximal $\ast $-ideal $M=M(I)$.

\begin{lemma}
\label{Lemma X4}A $\ast $-ideal $I$ of finite type is $\ast $-homogeneous if
and only if for each pair $X,Y$ of proper $\ast $-ideals of finite type
containing $I$ we have that $(X+Y)^{\ast }$ is proper.
\end{lemma}

\begin{proof}
Let $I$ be $\ast $-homogeneous, then any proper finite type $\ast $-ideals $%
X,Y$ containing $I$ are $\ast $-homogeneous contained in $M(I)$ and so $%
(X+Y)^{\ast }\subseteq M(I).$ Conversely if the condition holds and $I$ is
contained in two distinct maximal $\ast $-ideals $N_{1},N_{2}$. For $n\in
N_{1}\backslash N_{2}$ we have $(n,N_{2})^{\ast }=D,$ so there is a finite
set $J\subseteq N_{2}$ such that $(n,J)^{\ast }=D$, because $\ast $ is of
finite type. But then $X=(I,n)^{\ast }\subseteq N_{1}$ and $Y=(I,J)^{\ast
}\subseteq N_{2}$ both containing $I$ but $(X+Y)^{\ast }=D$ a contradiction.
\end{proof}

\begin{remark}
\label{Remark X5}Note that if $A$ and $B$ are proper $\ast $-ideals such
that $(A+B)^{\ast }=D$ and if $C$ is any proper $\ast $-ideal containing $B$
then $(A+C)^{\ast }=D,$ since $(A+C)^{\ast }=(A+B+C)^{\ast }.$
\end{remark}

\begin{theorem}
\label{Theorem X6}Let $\ast $ be a finite type star operation defined on an
integral domain $D.$ Then $D$ is of finite $\ast $-character if and only if
every $\ast $-ideal of finite type of $D$ is contained in at most a finite
number of mutually $\ast $-comaximal $\ast $-ideals of finite type.
\end{theorem}

\begin{proof}
(I) We first show that every $\ast $-ideal of finite type of $D$ is
contained in at least one $\ast $-homogeneous ideal of $D$. For suppose that
there is a $\ast $-ideal $A$ of finite type of $D$ that is not contained in
any $\ast $-homogeneous ideals of $D$. Then obviously $A$ is not $\ast $%
-homogeneous. So there are at least two proper $\ast $-ideals $A_{1},B_{1}$
of finite type such that $(A_{1}+B_{1})^{\ast }=D$ and $A\subseteq
A_{1},B_{1}$. Obviously, neither of $A_{1},B_{1}$ is homogeneous. As $B_{1}$
is not $\ast $-homogeneous there are at least two $\ast $-comaximal proper $%
\ast $-ideals $B_{11},B_{12}$ of finite type containing $B_{1}$. Now by
Remark \ref{Remark X5} $A_{1},B_{11},B_{12}$ are mutually $\ast $-comaximal
proper $\ast $-ideals containing $A$ and by assumption none of these is $%
\ast $-homogeneous. Let $B_{123}$ and $B_{22}$ be two $\ast $-comaximal
proper $\ast $-ideals containing $B_{12}.$ Then by Remark \ref{Remark X5}
and by assumption, $A_{1},B_{11},B_{22}$, $B_{123}$ are proper mutually $%
\ast $-comaximal $\ast $-ideals containing $A$ and none of these ideals is
homogeneous, and so on. Thus at stage $n$ we have a collection: $%
A_{1},B_{11},B_{22},...,$ $B_{nn},B_{12...n,n+1}$ that are proper mutually $%
\ast $-comaximal $\ast $-ideals containing $A$ and none of these ideals is
homogeneous. The process is never ending and has the potential of delivering
an infinite number of mutually $\ast $-comaximal proper $\ast $-ideals of
finite type containing $A,$ contrary to the finiteness condition. Whence the
conclusion.

Call two $\ast $-homogeneous ideals $A,B$ similar if $(A,B)^{\ast }\neq D,$
that is if $A$ and $B$ belong to the same maximal $\ast $-ideal. The
relation $R=$ \textquotedblleft $A$ is similar to $B$\textquotedblright\ is
obviously an equivalence relation on the set $S$ of $\ast $-homogeneous
ideals containing $A.$ Form a set $T$ of $\ast $-homogeneous ideals by
selecting one and exactly one $\ast $-homogeneous ideal from each
equivalence class of $R$. Then $T$ is a set of mutually $\ast $-comaximal $%
\ast $-homogeneous ideals containing $A$ and so must be finite because of
the finiteness condition. Let $|T|=n$ and claim that $n$ is the largest
number of mutually $\ast $-comaximal $\ast $-ideals of finite type
containing $A.$ For if not then there is say a set $U$ of mutually $\ast $%
-comaximal $\ast $-ideals of finite type that contain $A$ and $|U|=r>n.$
Then there is at least one member $B$ of $U$ that is $\ast $-comaximal with
each member of $T.$ (Since no two $\ast $-comaximal $\ast $-ideals share the
same maximal $\ast $-ideal.) But then, by (I), there is a $\ast $%
-homogeneous ideal $J$ containing $B.$ By Remark \ref{Remark X5}, $J$ is $%
\ast $-comaximal with each member of $T,$ yet by the construction of $T$ a $%
\ast $-homogeneous ideal containing $A$ must be similar to a member of $T,$
a contradiction. Finally if $P_{1},P_{2},...,P_{n}$ are maximal $\ast $%
-ideals such that each contains a member of $T$ then these are the only
maximal $\ast $-ideals containing $A.$ For if not then there is a maximal $%
\ast $-ideal $M\neq P_{i}$ containing $A$ and there is $x\in M\backslash
P_{i},$ $i=1,2,...,n.$ But then $(A,x)^{\ast }$ is a finite type $\ast $%
-ideal containing $A$ and $\ast $-comaximal with each member of $T,$ yet by
(I) $(A,x)^{\ast }$ must be contained in a $\ast $-homogeneous ideal that is 
$\ast $-comaximal with each member of $T$, a contradiction. For the converse
note that if a nonzero non unit $x\in D$ is contained in infinitely many
mutually $\ast $-comaximal ideals then $D$ cannot be of finite $\ast $%
-character, because a maximal $\ast $-ideal cannot contain two or more $\ast 
$-comaximal ideals.
\end{proof}

So, if we must construct a $\ast $-homogeneous ideal we know where to go.
Otherwise there are plenty of $\ast $-potent domains, with one kind studied
in \cite{HZ 2019} under the name $\ast $-super potent domains. Let's note
here that there is a slight difference between the definitions. Definition
1.1 of \cite{HZ 2019} calls a finitely generated ideal $I$ $\ast $-rigid if $%
I$ is contained in a unique maximal $\ast $-ideal. But it turns out that if $%
I$ is $\ast $-rigid, then $I^{\ast }$ is $\ast $-homogeneous and if $J$ is $%
\ast $-homogeneous then $J$ contains a finitely generated ideal $K$ such
that $K$ is exactly in the same maximal $\ast $-ideal containing $J,$ making 
$K$ $\ast $-rigid, see also \cite{Zaf2 2019}.)

\bigskip

\section{\label{S2}What $\ast $-homogeneous ideals can do}

This much about $\ast $-homogeneous ideals and potent domains leads to the
questions: What else can $\ast $-homogeneous ideals do? $\ast $-homogeneous
ideals arise and figure prominently in the study of finite $\ast $-character
of integral domains. The domains of $\ast $-finite character where the $\ast 
$-homogeneous ideals show their full force are the $\ast $-Semi Homogeneous (%
$\ast $-SH) Domains.

It turns out, and it is easy to see, that if $I$ and $J$ are two $\ast $%
-homogeneous ideals that are similar, i.e. that belong to the same unique
maximal $\ast $-ideal (i.e. $M(I)=M(J)$ in the notation and terminology of 
\cite{AZ 2019}) then $(IJ)^{\ast }$ is $\ast $-homogeneous belonging to the
same maximal $\ast $-ideal. With the help of this and some auxiliary results
it can then be shown that if an ideal $K$ is a $\ast $-product of finitely
many $\ast $-homogeneous ideals then $K$ can be uniquely expressed as a $%
\ast $-product of mutually $\ast $-comaximal $\ast $-homogeneous ideals.
Based on this a domain $D$ is called a $\ast $-semi homogeneous ($\ast $-SH)
domain if every proper principal ideal of $D$ is expressible as a $\ast $%
-product of finitely many $\ast $-homogeneous ideals. It was shown in \cite[%
Theorem 4]{AZ 2019} that $D$ is a $\ast $-SHD if and only if $D$ is a $\ast $%
-h-local domain ($D$ is a locally finite intersection of localizations at
its maximal $\ast $-ideals and no two maximal $\ast $-ideals of $D$ contain
a common nonzero prime ideal.) Now if we redefine a $\ast $-homogeneous
ideal so that the $\ast $-product of two similar, newly defined, $\ast $%
-homogeneous ideals is a $\ast $-homogeneous ideal meeting the requirements
of the new definition, we have a new theory.

To explain the process of getting a new theory of factorization merely by
producing a suitable definition of a $\ast $-homogeneous ideal we give below
one such theory.

Let's recall first that if $A=(a_{1},...,a_{m})$ is a finitely generated
ideal then $A_{(r)}$ denotes $(a_{1}^{r},...,a_{m}^{r})$. Let's also recall
that if $A$ is $\ast $-invertible then $(A^{r})^{\ast
}=(a_{1}^{r},...,a_{m}^{r})^{\ast }$ \cite[Lemma 1.14]{AHZ 2019}.

\begin{definition}
\label{Definition X7} Call a $\ast $-homogeneous ideal $I$ $\ast $-almost
factorial general homogeneous ($\ast $-afg homogeneous) if (afg1) $I$ is $%
\ast ~$-invertible, and (afg2) for each finite type $\ast $-homogeneous
ideal $J\subseteq M(I)$ we have for some $r\in N,$ $(I^{r}+J)_{m}^{\ast }$
is principal for some $m\in N,$ $m$ depending upon the choice of generators
of $(I^{r}+J).$
\end{definition}

(You can also redefine it as: Definition A. Call a $\ast $-homogeneous ideal 
$I$ $\ast $-almost factorial general homogeneous ($\ast $-afg homogeneous)
if (afg1) $I$ is $\ast ~$-invertible and (afg2) for each finitely generated $%
\ast $-homogeneous ideal $J\subseteq M(I)$ such that $J^{\ast }\supseteq
I^{r}$ , for some $r\in N$, we have $(J)_{m}^{\ast }$ principal for some $%
m\in N$. (Here you may add that $m$ may vary with each choice of generators
of $J.$ And redo the following accordingly.)

\begin{lemma}
\label{Lemma X8} Let $I$ be $\ast $-invertible and $J$ any f.g. ideal then $%
((IJ)_{r})^{\ast }=(I^{r}J_{r})^{\ast }$
\end{lemma}

\begin{proof}
Let $I=(a_{1},...,a_{m}),$ $J=(b_{1},...,b_{n}).$ Then $IJ=(%
\{a_{i}b_{j}|i=1,...,m$; $j=1,...n\})$ and $(IJ)_{r}^{\ast
}=\{a_{i}^{r}b_{j}^{r}|i=1,...,m$; $j=1,...n\}^{\ast
}=((a_{1}^{r},...,a_{m}^{r})(b_{1}^{r},...,b_{n}^{r}))^{\ast
}=(I^{r}J_{r})^{\ast },$ because $I$ is $\ast $-invertible.
\end{proof}

Using the above definition, we can be sure of the following.

\begin{proposition}
\label{Proposition X9} The following hold for a $\ast $-afg ideal $I.$ (1) $%
(I^{r})^{\ast }$ is principal for some positive integer $r,$ (2) for any
finitely generated ideal $J\subseteq M(I),$ we have $(I^{m})^{\ast
}\subseteq (J_{m})^{\ast }$ or $(J_{m})^{\ast }\subseteq (I^{m})^{\ast }$
for some positive integer $m,$ (3) if $J$ is a $\ast $-invertible $\ast $%
-ideal that contains $I,$ then $J$ itself is a $\ast $-afg ideal and (4) if $%
J$ is a $\ast $-afg ideal similar to $I$ (i.e., $J\subseteq M(I)),$ then $%
(IJ)^{\ast }$ is $\ast $-afg similar to both $I$ and $J.$
\end{proposition}

\begin{proof}
(1) If $I$ is $\ast $-afg, $(I^{r})^{\ast }=(f),$ for some $r\in N$ and $%
f\in D$ by definition and we can choose $r$ to be minimum. ($(I+I)_{r}=(f)$)

(2) By definition, if $I$ is $\ast $-afg, we also have ($(I^{m}+J_{m})^{\ast
}=(d),$ for each finitely generated ideal $J.$ Dividing both sides by $d$ we
get $(I^{m}/d+J_{m}/d)^{\ast }=D.$ Now as $I^{m}$ and $J_{m}$ are contained
in $M(I),$ and no other maximal $\ast $-ideal, so $(I^{m}/d)^{\ast }$ and $%
(J_{m}/d)^{\ast }$ have no choice but to be in $M(I),$ if non-trivial. So, $%
(I^{m}/d)^{\ast }=D$ or $(J_{m}/d)^{\ast }=D.$ Thus if $(I^{m})^{\ast }=dD,$
then $(I^{m})^{\ast }=(d)\supseteq (J_{m})^{\ast }$ and if $(J_{m})^{\ast
}=dD$ then $I^{m}\subseteq J_{m}=dD.$ Thus by (afg2) $(I^{m})^{\ast }$ is
principal and contains $(J_{m})^{\ast }$ or $(J_{m})^{\ast }$ is principal
and contains $(I^{m})^{\ast },$ for some $m\in N.$

(3) Note that as $J\supseteq I$ we have $J^{r}\supseteq I^{r}$ for all
positive integers $r.$ Next for every finitely generated ideal $F$ such that 
$F^{\ast }\supseteq J$ $^{s}$ for some $s$ we have $F^{\ast
}=(J^{s}+G)^{\ast }=(J^{s}+I^{s}+G)^{\ast }=(I^{s}+(J^{s}+G))^{\ast }$ and
so $(F_{m})^{\ast }=(d)$ for some positive integer $m$ and for each $\ast $%
-homogeneous ideal $G.)$ (4) If $I,J$ are two similar $\ast $-afg
homogeneous ideals then $(IJ)^{\ast }$ is similar to both $I$ and $J.${} $%
(IJ)^{\ast }$ is $\ast $-invertible and $\ast $-homogeneous and of course
similar to both $I$ and $J.$ We have to show that for each $\ast $%
-homogeneous ideal $G,$ for some $r\in N,$ $(I^{r}J^{r}+G)_{m}^{\ast }=(d)$
for some $m.$ Let $F=(I^{r}J^{r}+G)^{\ast }.$ By (2) we know that $%
I^{m}\supseteq J^{m}$ or $J^{m}\supseteq I^{m},$ say $I^{m}\supseteq J^{m}.$
Now consider $(F_{m})^{\ast }=(I^{mr}J^{mr}+G_{m})^{\ast }\supseteq
(J^{2mr}+G_{m})$ or $(F_{m})^{\ast }=(J^{2mr}+H)^{\ast }$ and by definition $%
(F_{mt})^{\ast }$ is principal for some $t.$
\end{proof}

Now define a $\ast $-afg semi homogeneous domain ($\ast $-afg-SHD) as: $D$
is a $\ast $-afg-SHD if every nonzero non unit of $D$ is expressible as a $%
\ast $-product of finitely many $\ast $-afg homogeneous ideals. Indeed $D$
is a $\ast $-afg-SHD is a $\ast $-SHD whose $\ast $-homogeneous ideals are $%
\ast $-afg homogeneous. (S. Xing, a student of Wang Fanggui, is working with
me on this topic. Xing, incidentally, is also at Chengdu University, China.
Now Dan Anderson has also joined in and there's a possibility that the
definition will be completely twisted out of shape.)

Next, each of the definitions of homogeneous elements can actually give rise
to $\ast $-potent domains in the same manner as the $\ast $-super potent
domains of \cite{HZ 2019}. In \cite{HZ 2019}, for a star operation $\ast $
of finite character, a $\ast $-homogeneous ideal is called $\ast $-rigid.
The $\ast $-maximal ideal containing a $\ast $-homogeneous ideal $I$ may be
called a $\ast $-potent maximal $\ast $-ideal, as we have already done. Next
we may call the $\ast $-homogeneous ideal $I$ $\ast $-super-homogeneous if
each $\ast $-homogeneous ideal $J$ containing $I$ is $\ast $-invertible and
we may call a $\ast $-potent domain $D$ $\ast $-super potent if every
maximal $\ast $- ideal $I$ of $D$ contains a $\ast $-super homogeneous
ideal. But then one can study $\ast $-A-potent domains where A refers to a $%
\ast $-homogeneous ideal that corresponds to a particular definition. For
example a $\ast $-homogeneous ideal $J$ is said to be of type $1$ in \cite%
{AZ 2019} if $\sqrt{J}=M(J).$ So we can talk about $\ast $-type $1$ potent
domains as domains each of whose maximal $\ast $-ideals contains a $\ast $%
-homogeneous ideal of type $1.$ The point is, to each suitable definition
say A of a $\ast $-homogeneous ideal we can study the $\ast $-A-potent
domains as we studied the $\ast $-super potent domains in \cite{HZ 2019}. Of
course the theory corresponding to definition A would be different from that
of other $\ast $-potent domains. For example each of the maximal $\ast $%
-ideal of the $\ast $-type $1$ potent domain would be the radical of a $\ast 
$-homogeneous ideal etc. Now as it is usual we present some of the concepts
that have some direct and obvious applications, stemming from the use of $%
\ast $-homogeneous ideals. For this we select the $\ast $-f-potent domains
for a study.

\subsection{$\ast $-f-potent domains}

Let $\ast $ be a finite type star operation defined on an integral domain $D$%
. Call a nonzero non unit element $r$ of $D$ $\ast $-factorial rigid ( $\ast 
$-f-rigid) if $r$ belongs to a unique maximal $\ast $-ideal and every finite
type $\ast $-homogeneous ideal containing $r$ is principal. Indeed if $r$ is
a $\ast $-f-rigid element then $rD$ is a $\ast $-f- homogeneous ideal and
hence a $\ast $-super homogeneous ideal. So the terminology and the theory
developed in \cite{AZ 2019} applies. Note here that every non unit factor $s$
of a $\ast $-f-rigid element $r$ is $\ast $-f-rigid because of the
definition. Note also that if $r,s$ are similar $\ast $-f-rigid elements
(i.e. $rD,$ $sD$ are similar $\ast $-f-homogeneous ideals) then $rs$ is a $%
\ast $-f-rigid element similar to $r$ and $s$ and so if $r$ is $\ast $%
-f-rigid then $r^{n}$ is $\ast $-f-rigid for any positive integer $n.$

\begin{example}
\label{Example A}. Every prime element is a $t$-f-rigid element.
\end{example}

Call a maximal $\ast $-ideal $M$ $\ast $-f-potent if $M$ contains a $\ast $%
-f-rigid element and a domain $D$ $\ast $-f-potent if every maximal $\ast $%
-ideal of $D$ is $\ast $-f-potent.

\begin{example}
\label{Examples B}. UFDs PIDs, Semirigid GCD domains, prime potent domains
are all $t$-f-potent.
\end{example}

(domains in which every maximal $t$-ideal contains a prime element may be
called prime potent. Indeed a prime element generates a maximal $t$-ideal 
\cite[13.5]{H-K 1998}. (So a domain in which every maximal $t$-ideal
contains a prime element is simply a domain in which every maximal $t$-ideal
is principal.)

The definition suggests right away that if $r$ is $\ast $-f-rigid and $x$
any element of $D$ then $(r,x)^{\ast }=sD$ for some $s\in D$ and applying
the $v$-operation to both sides we conclude that $GCD(r,x)=(r,x)_{v}$ of $r$
exists with every nonzero element $x$ of $D$ and that for each pair of
nonzero factors $u,v$ of $r$ we have $u|v$ or $v|u$; that is $r$ is a rigid
element of $D$, in Cohn's terminology \cite{Coh 1968}. Indeed it is easy to
see, if necessary with help from \cite{AZ 2019}, that a finite product of $%
\ast $-f-rigid elements is uniquely expressible as a product of mutually $%
\ast $-comaximal $\ast $-f-rigid elements, up to order and associates and
that if every nonzero non unit of $D$ is expressible as a product of $\ast $%
-f-rigid elements then $D$ is a semirigid GCD domain of \cite{Zaf 1975}.
Also, as we shall show below, a $t$-f-potent domain of $t$-dimension one
(i.e. every maximal $t$-ideal is of height one) is a GCD domain of finite $t$%
-character. But generally a $t$-f-potent domain is far from being a GCD
domain. Before we delve into examples, let's prove a necessary result, by
mimicking Theorem 4.12 of \cite{CMZ 1978} and its proof. (We shall also use
Theorem 4.21 of \cite{CMZ 1978}, in the proofs of results below.)

\begin{proposition}
\label{Proposition B1} Let $D$ be an integral domain and let $L$ be an
extension of the field of fractions $K$ of $D.$ Then each ideal $I$ of $%
R=D+XL[X]$ is of the form $f(X)FR=f(X)(F+XL[X])$, where $F$ is a nonzero $D$%
-submodule of $L$ such that $f(0)F\subseteq D$ and $f(X)\in L[X]$. The
finitely generated ideals of $R$ are of the form $f(X)JR$, where $J$ is a
finitely generated $D$-submodule of $L$ and $f(X)\in R$.
\end{proposition}

\begin{proof}
First observe that a subset of $R$ of the form $f(X)FR,$ where $%
f(0)F\subseteq D,$ is in fact an ideal of $R$. According to \cite[Lemma 1.1]%
{CMZ 1986}, the following are equivalent for an ideal $I$ of $R$: (1) $I$ is
such that $I\cap D\neq 0,$ (2) $I\supseteq XL[X]$ and (3) $IL[X]=L[X]$.
Further if any of these hold, then $I=I\cap D+XL[X]=(I\cap D)R$ and taking $%
f=1$, $F=I\cap D$ we have the stated form. Let's now consider the case when $%
IL[X]\neq L[X].$ In this case $IL[X]=f(X)L[X]$ where $f(X)$ is a variable
polynomial of $L[X].$ Then there is a nonzero element $\alpha \in L$ such
that $\alpha f(X)\in I$. Let $F=\{\alpha \in L$ $|$ $\alpha f(X)\in I\}$.
Then $F$ is a $D$-submodule of $L$. Since $F\neq 0$ and $f(X)F\subseteq
I,I\supseteq f(X)FR=f(X)(F+XL[X]).$ Now if $h(X)\in I$, then $%
h(X)=f(X)(\alpha _{0}+...+\alpha _{n}X^{n})$, where $\alpha _{0},,...,\alpha
_{n}\in L$ whence $h(X)=$ $\alpha _{0}f(X)+h\prime (X)$, where $h\prime
(X)\in f(X)XL[X]\in I$. Hence $\alpha _{0}\in F$ and $h(X)\in f(X)(F+XL[X])$%
. Thus $I=f(X)(F+XL[X])=f(X)FR$, from which it also follows that $%
f(0)F\subseteq D$. Finally let $I$ be finitely generated, then by the above
we have $I=f(X)FR$ where $F$ is a finitely generated $D$-submodule of $L$
and $f(X)\in L[X].$ If $f(0)=0,$ then $f(X)$ is obviously in $R.$ So let's
consider $f(0)=h\neq 0$ and $F=(\alpha _{1},\alpha _{2},...,\alpha _{r})D$.
Since $f(0)F\subseteq D$ we must have $h\alpha _{i}\in D.$ But then $%
I=f(X)FR $ can be written as $I=\frac{f(X)}{h}(h\alpha _{1},h\alpha
_{2},...,h\alpha _{r})R$ where $\frac{f(X)}{h}\in R.$
\end{proof}

(I was struggling with an earlier version of Proposition \ref{Proposition B1}
and Prof. T. Dumitrescu's suggested improvement for it when I remembered
Theorem 4.12 of \cite{CMZ 1978}. I am thankful for his input.)

\begin{lemma}
\label{Lemma B2} Let $D$ be an integral domain and let $L$ be an extension
field of the field of fractions $K$ of $D.$ Then $d\in D\backslash (0)$ is a 
$t$-f-homogeneous element of $D$ if and only if $d$ is a $t$-f-homogeneous
element of $D+XL[X]$.
\end{lemma}

\begin{proof}
Let's first note that $D+XL[X]$ has the $D+M$ form. Thus if $I$ is a nonzero
ideal of $D$ then $(I+XL[X])_{v}=I_{v}+XL[X]=I_{v}(D+XL[X]),$ by \cite[%
Proposition 2.4]{AR 1988} and using this we can also conclude that $%
(I+XL[X])_{t}=I_{t}+XL[X]=I_{t}(D+XL[X]).$ Now let $d$ be a $t$%
-f-homogeneous element of $D$ then $dD$ is a $t$-f-homogeneous ideal, so any 
$t$-ideal of finite type, of $D,$ containing $dD$ is principal. Next
consider $d\in D+XL[X].$ Any $t$-ideal of finite type $F$ of $R$ containing $%
d$ intersects $D$ and so has the form $(F\cap D)+XL[X]$, according to \cite[%
Lemma 1.1]{CMZ 1986}. Consequently $F$ contains $dD+XL[X].$ We show that $F$
is principal. For this let $%
F=(a_{1}+Xf_{1}(X),...,a_{n}+Xf_{n}(X)_{t}=((a_{1},...,a_{n})+XL[X])_{v}$ $%
=((a_{1},...,a_{n})_{v}+XL[X]).$ But $((a_{1},...,a_{n})_{v}+XL[X])=F%
\supseteq dD+XL[X]$ forces $(a_{1},...,a_{n})_{v}\supseteq dD$. Also $dD$
being $t$-f-rigid, $(a_{1},...,a_{n})_{v}$ must be principal whence $F$ is
principal Now note that according to \cite{CMZ 1986}, every prime ideal $M$
of $R$ that intersects $D$ is of the form $M\cap D+XL[X]$ and using the
above mentioned result of \cite[Proposition 2.4]{AR 1988} we can show that
every maximal $t$-ideal $M$ that intersects $D$ is of the form $M\cap
D+XL[X] $ where $M\cap D$ is a maximal $t$-ideal of $D$ and that,
conversely, if $m$ is a maximal ideal of $D$ then $m+XL[X]$ is a maximal
ideal of $R.$ Thus, finally, if $m$ is the unique maximal $t$-ideal of $D$
containing $d$ then $m+XL[X]$ is a maximal $t$-ideal of $R$ containing $d$
and if $N$ were another maximal $t$-ideal containing $d$ then $N\cap D$
would be another maximal $t$-ideal of $D$ containing $d$ a contradiction.
Thus $d$ is a $t$-f-homogeneous ideal of $R.$
\end{proof}

\begin{proposition}
\bigskip \label{Proposition B3} Let $D$ be an integral domain and let $L$ be
an extension field of the field of fractions $K$ of $D.$ Then $D$ is $t$%
-potent if and only if $R=D+XL[X]$ is.
\end{proposition}

\begin{proof}
Note that, according to \cite[Lemma 1.2]{CMZ 1986}, every prime ideal $P$ of 
$R$ that is not comparable with $XL[X]$ contains an element of the form $%
1+Xg(X)$, so must contain a prime element of the form $1+Xg(X)$ and so must
be a principal prime. We next show that a finitely generated ideal $%
F\nsubseteq XL[X]$ of $R$ is $t$-homogeneous if and only if $F$ is of the
form $A+XL[X],$ where $A$ is a $t$-homogeneous ideal of $D$ or generated by
a prime power of the form $(1+Xh(X))^{n},$ \cite[Theorem 4.21]{CMZ 1978}.
Obviously if $A$ is contained in a unique maximal $t$-ideal $P$ of $D$ then $%
A+XL[X]$ is contained in the maximal $t$-ideal $P+XL[X]$ and any maximal $t$%
-ideal that contains $A+XL[X]$ also contains $P+XL[X].$ Next, an ideal
generated by a prime power is $t$-homogeneous anyway. Conversely let $F$ be
a finitely generated nonzero ideal of $R.$ Then by Proposition \ref%
{Proposition B1}, $F=f(X)(J+XL[X])$ where $f(X)\in L[X]$ as $F$ is not
contained in $XL[X],$ $f(0)=1$ forcing $J$ to be a finitely generated ideal
of $D.$ If in addition $F$ has to be $t$-homogeneous then $F$ is either
contained in a prime ideal of the form $P+XL[X]$ or in a prime ideal
incomparable with $XL[X].$ In the first case $F=J+XL[X]$ where $J$ is a
rigid ideal belonging to $P$ and in the second case $F=f(X)R,$ \cite[Theorem
4.21]{CMZ 1978}.
\end{proof}

\begin{corollary}
\label{Corollary B4} Let $D$ be an integral domain and let $L$ be an
extension field of the field of fractions $K$ of $D.$ Then $D$ is $t$%
-f-potent if and only if $R=D+XL[X]$ is.
\end{corollary}

\begin{proof}
Suppose that $D$ is $t$-f-potent. As in the proof of Proposition \ref%
{Proposition B3} every maximal $t$-ideal $P$ of $R$ that is not comparable
with $XL[X]$ contains an element of the form $1+Xg(X)$, so must contain a
prime element of the form $1+Xg(X)$ and so must be a principal prime. Next
the maximal $t$-ideals comparable with $XL[X]$ are of the form $P+XL[X]$
where $P$ is a maximal $t$-ideal of $D.$ Since $D$ is $t$-f-potent $P$
contains a $t$-f-rigid element, which is also a $t$-f-rigid element of $R,$
by Lemma \ref{Lemma B2}. So $P+XL[X]$ contains a $t$-f-rigid element of $R.$
In sum, every maximal $t$-ideal of $R$ contains a $t$-f-rigid element of $R$
and so $R$ is $t$-f-potent. Conversely suppose that $R$ is $t$-f-potent.
Then as for each maximal $t$-ideal $M$ of $D,$ $M+XL[X]$ is a maximal $t$%
-ideal, each $M$ contains a $t$-f-rigid element of $R$ and hence of $D,$ by
Lemma \ref{Lemma B2}. Thus each maximal $t$-ideal of $D$ contains a $t$%
-f-rigid element of $D.$
\end{proof}

Recall, from \cite{AAZ 1995}, that a GCD domain of finite $t$-character that
is also of $t$-dimension $1$ is termed as a generalized UFD (GUFD).

\begin{example}
\label{Example C} If $D$ is a UFD (GUFD, Semirigid GCD domain) and $L$ an
extension of the quotient field $K$ of $D,$ then the ring $D+XL[X]$ is a $t$%
-f-potent domain.
\end{example}

The $t$-f-potent domains and their examples are nice but we must show that
they have some useful properties. We start with the most striking property.
Here let $X$ be an indeterminate over $D.$ A polynomial $f=\sum $ $%
a_{i}X^{i} $ is called primitive if its content $%
A_{f}=(a_{0},a_{1},...,a_{n})$ generates a primitive ideal, i.e., $%
(a_{0},a_{1},...,a_{n})\subseteq aD$ implies $a$ is a unit and super
primitive if $(A_{f})_{v}=D.$ It is known that while a super primitive
polynomial is primitive a primitive polynomial may not be super primitive,
see e.g. Example 3.1 of \cite{AZ 2007}. A domain $D$ is called a PSP domain
if each primitive polynomial $f$ over $D$ is superprimitive, i.e. if $%
(A_{f})_{v}=D.$

\begin{proposition}
\label{Proposition E} A $t$-f-potent domain $D$ has the PSP property.
\end{proposition}

\begin{proof}
Let $f=\sum $ $a_{i}X^{i}$ be primitive i.e. $(a_{0},a_{1},...,a_{n})%
\subseteq aD$ implies $a$ is a unit and consider the finitely generated
ideal $(a_{0},a_{1},...,a_{n})$ in a $t$-f- potent domain$D.$ Then $%
(a_{0},a_{1},...,a_{n})$ is contained in a maximal $t$-ideal $M$ associated
with a $t$-f-rigid element $r$ (of course $M=M(rD))$ if and only if $%
(a_{0},a_{1},...,a_{n},r)_{t}=sD\neq D.$ Since every maximal $t$-ideal of a $%
t$-f-potent domain is associated with a $t$-f-rigid element, we conclude
that in a $t$-f-potent domain $D,$ $f=\sum $ $a_{i}X^{i}$ primitive implies
that $A_{f}$ is contained in no maximal $t$-ideal of $D;$ giving $%
(A_{f})_{v} $ $=D$ which means that each primitive polynomial $f$ in a $t$%
-f-potent domain $D$ is actually super primitive.
\end{proof}

Now PSP implies AP i.e. every atom is prime, see e.g. \cite{AZ 2007}. So, in
a $t$-f-potent domain every atom is a prime. If it so happens that a $t$%
-f-potent domain has no prime elements then the $t$-f-potent domain in
question is atomless. Recently atomless domains have been in demand. The
atomless domains are also known as antimatter domains. Most of the examples
of atomless domains that were constructed were the so-called pre-Schreier
domains, i.e. domains in which every nonzero non unit $a$ is primal (is such
that ($a|xy$ implies $a=rs$ where $r|x$ and $s|y)$. One example (Example
2.11 \cite{AZ 2007}) was laboriously constructed in \cite{AZ 2007} and this
example was atomless and not pre-Schreier, As we indicate below, it is easy
to establish a method of telling whether a $t$-f-potent domain is
pre-Schreier or not.

Cohn in \cite{Coh 1968} called an element $c$ in an integral domain $D$
primal if (in $D)$ $c|a_{1}a_{2}$ implies $c=c_{1}c_{2}$ where $c_{i}|a_{i}.$
Cohn \cite{Coh 1968} assumes that $0$ is primal. We deviate slightly from
this definition and call a nonzero element $c$ of an integral domain $D$
primal if $c|a_{1}a_{2},$ for all $a_{1},a_{2}\in D\backslash \{0\},$
implies $c=c_{1}c_{2}$ such that $c_{i}|a_{i}.$ He called an integral domain 
$D$ a Schreier domain if (a) every (nonzero) element of $D$ is primal and
(b) $D$ is integrally closed. We have included nonzero in brackets because
while he meant to include zero as a primal element, he mentioned that the
group of divisibility of a Schreier domain is a Riesz group. Now the
definition of the group of divisibility $G(D)(=\{\frac{a}{b}D:$ $a,b\in
D\backslash \{0\}\}$ ordered by reverse containment) of an integral domain $%
D $ involves fractions of only nonzero elements of $D$, so it's permissible
to restrict primal elements to be nonzero and to study domains whose nonzero
elements are all primal. This is what McAdam and Rush did in \cite{MR 1978}.
In \cite{Zaf 1987} integral domains whose nonzero elements are primal were
called pre-Schreier. It turned out that pre-Schreier domains possess all the
multiplicative properties of Schreier domains. So let's concentrate on the
terminology introduced by Cohn as if it were actually introduced for
pre-Schreier domains.

Cohn called an element $c$ of a domain $D$ completely primal if every factor
of $c$ is primal and proved, in Lemma 2.5 of \cite{Coh 1968} that the
product of two completely primal elements is completely primal and stated in
Theorem 2.6 a Nagata type result that can be rephrased as: Let $D$ be
integrally closed and let $S$ be a multiplicative set generated by
completely primal elements of $D$. If $D_{S}$ is a Schreier domain then so
is $D.$ This result was analyzed in \cite{AZ 2007} and it was decided that
the following version (\cite[Theorem 4.4]{AZ 2007} of Cohn's Nagata type
theorem works for pre-Schreier domains.

\begin{theorem}
\label{Theorem F} (Cohn's Theorem for pre-Schreier domains). Let $D$ be an
integral domain and $S$ a multiplicative set of $D$. (i) If $D$ is
pre-Schreier, then so is $D_{S}$ . (ii) (Nagata type theorem) If $D_{S}$ is
a pre-Schreier domain and $S$ is the set generated by a set of completely
primal elements of $D$, then $D$ is a pre-Schreier domain.
\end{theorem}

Now we have already established above that if $r$ is a $t$-f-rigid element
then $(r,x)_{v}$ is principal for each $x\in D\backslash \{0\}.$ But then $%
(r,x)_{v}$ is principal for each $x\in D\backslash \{0\}$ if and only if $%
(r)\cap (x)$ is principal for each $x\in D\backslash \{0\}.$ But then $r$ is
what was called in \cite{AZ 1995} an extractor. Indeed it was shown in \cite%
{AZ 1995} that an extractor is completely primal. Thus we have the following
statement.

\begin{corollary}
\label{Corollary G} Let $D$ be a $t$-f-potent domain. Then $D$ is
pre-Schreier if and only if $D_{S}$ is pre-Schreier for some multiplicative
set $S$ that is the saturation of a set generated by some $t$-f rigid
elements.
\end{corollary}

(Proof. If $D$ is pre-Schreier then $D_{S}$ is pre-Schreier anyway. If on
the other hand $D_{S}$ is pre-Schreier and $S$ is (the saturation of a set)
multiplicatively generated by some $t$-f- rigid elements. Then by Theorem %
\ref{Theorem F}, $D$ is pre-Schreier.)

One may note here that if $D_{S}$ is not pre-Schreier for any multiplicative
set $S,$ then $D$ is not pre-Schreier. So the decision making result of Cohn
comes in demand only if $D_{S}$ is pre-Schreier. Of course in the Corollary %
\ref{Corollary G} situation, the saturation $S$ of the multiplicative set
generated by all the $t$-f-rigid elements of $D,$ leading to: if $D_{S}$ is
not pre-Schreier then $D$ is not pre-Schreier for sure and if $D_{S}$ is
pre-Schreier then $D$ cannot escape being a pre- Schreier domain.

\begin{example}
\label{Example H} Let $D=\cap _{i=1}^{i=n}V_{i}$ be a finite intersection of
distinct non-discrete rank one valuation domains with quotient field $%
K=qf(D),$ $X$ an indeterminate over $D$ and let $L$ be a proper field
extension of $K.$ Then (a) $D+XL[X]$ is a non-pre-Schreier, $t$-f-potent
domain and (b) $D+XL[X]_{(X)}$ is an atomless non-pre-Schreier, $t$-f-potent
domain.
\end{example}

Illustration: (a) It is well known that $D$ is a Bezout domain with exactly $%
n$ maximal ideals, $M_{i}$ \cite{Kap 1970}, with $V_{i}=D_{M_{i}}.$ Thus $D$%
= $\cap D_{M_{i}}$ and each of $M_{i}$ being a $t$-ideal must, each,
contains a $t$-homogeneous ideal by Proposition \ref{Proposition X1}. $%
D+XL[X]$ is $t$-f-potent by Corollary \ref{Corollary B4}.

One more result that can be added needs introduction to a neat construction
called the Nagata ring construction these days. This is how the construction
goes.

Let $\ast $ be a star operation on a domain $D$, let $X$ be an indeterminate
over $D$ and Let $S_{\ast }=\{f\in D[X]|$ $(A_{f})^{\ast }=D\}.$ Then the
ring $D[X]_{S_{\ast }}$ is called the Nagata construction from $D$ with
reference to $\ast $ and is denoted by $Na(D,\ast ).$ Indeed $Na(D,\ast
)=Na(D,\ast _{f})$

\bigskip

\begin{proposition}
\label{Proposition G1} (\cite{Kan 1989} Proposition 2.1.) Let $\ast $ be a
star operation on $R$. Let $\ast _{f}$ be the finite type star operation
induced by $\ast .$ Let $S_{\ast }=\{f\in D[X]|(A_{f})_{\ast }=D\}$. Then
(1) $S_{\ast }=D[X]\backslash \cup _{M\in \Gamma }M[X]$ where $\Gamma $ is
the set of all maximal $\ast _{f}$-ideals of $D$. (Hence $S_{\ast }$ is a
saturated multiplicatively closed subset of $D[X].$), (2) $\{M[X]_{S_{\ast
}}\}$ is the set of all maximal ideals of $[DX]_{S_{\ast }}.$
\end{proposition}

\bigskip

As pointed out in \cite{FHP 2019}, proof of Part (1) of the following
proposition has a minor flaw, in that for a general domain it uses a result (%
\cite[38.4]{Gil 1972}) that is stated for integrally closed domains. The fix
offered in \cite{FHP 2019} is a new result and steeped in semistar
operations. We offer, in the following, a simple change in the proof of \cite%
[(1) Proposition 2.2.]{Kan 1989} to correct the flaw indicated above.

\bigskip

\begin{proposition}
\label{Proposition G2} (\cite{Kan 1989} Proposition 2.2.) Let $T$ be a
multiplicatively closed subset of $D[X]$ contained in $S_{v}=\{f\in
D[X]|(A_{f})_{v}=D\}$. Let $I$ be a nonzero fractional ideal of $D$. Then
(1) $(I[X]_{T})^{-1}=I^{-1}[X]_{T}$, (2) $(I[X]_{T})_{v}=I_{v}[X]_{T}$ and
(3) $(I[X]_{T})_{t}$ $=I_{t}[X]_{T}.$
\end{proposition}

(1) It is clear that $I^{-1}[X]_{T}\subseteq (I[X]_{T})^{-1}$. Let $u\in
(I[X]_{T})^{-1}.$ Since for any $a\in I\backslash \{0\}$ we have $%
(I[X]_{T})^{-1}\subseteq a^{-1}D[X]_{T}\subseteq K[X]_{T}$ we may assume
that $u=f/h$ with $f\in K[X]$ and $h\in T$. Then $f\in (I[X]_{T})^{-1}$.
Hence $fI[X]_{T}$ $\subseteq D[X]_{T}$ $.$ Hence $bf\in D[X]_{T}$ for any $%
b\in I$. Now $bfg\in $ $D[X]$ for some $g\in S_{v}$. So $(A_{bfg})_{v}%
\subseteq D$. By \cite[Proposition 2.2.]{MNZ 1990}, $%
(A_{bfg})_{v}=(A_{bf}A_{g})_{v}=(A_{bf})_{v}$, since $(A_{g})_{v}=D$ and
hence $v$-invertible. Therefore $bA_{f}\subseteq
(bA_{f})_{v}=(A_{bf})_{v}\subseteq D$ for any $b\in I$ Hence$A_{f}\subseteq
I^{-1}$. Hence $f\in I^{-1}[X]$ and $f/h\in I^{-1}[X]_{T}.$ Therefore $%
(I[X]_{T})^{-1}=I^{-1}[X]_{T}.$

\begin{theorem}
\label{Theorem G3} (\cite{Kan 1989}, Theorem 2.4.) Let $\ast $ be a finite
type star operation on $D$. Let $I$ be a$F[X]_{S_{v}}$ nonzero ideal of $D$.
Then $I$ is $\ast $-invertible if and only if $I[X]_{S_{\ast }}$ is
invertible.
\end{theorem}

\begin{theorem}
\label{Theorem G4}(\cite{Kan 1989}, Proposition 2.14.) Let $\ast $ be a star
operation on $D$. Then any invertible ideal of $D[X]_{S_{\ast }}$ is
principal.
\end{theorem}

Thus we have the following corollary.

\begin{corollary}
\label{Corollary G5} Let $I$ be a $t$-ideal of finite type of $D$. Then $I$
is $t$-invertible if and only if $I[X]_{S_{v}}$ is principal.
\end{corollary}

\begin{proof}
If $I[X]_{S_{v}}$ is principal then $I[X]_{S_{v}}$ is invertible and so $I$
is $t$-invertible by Theorem \ref{Theorem G3}. Conversely let $F$ be a
finitely generated ideal such that $F_{t}=I.$ Then $F$ is $t$-invertible and
so, by Theorem \ref{Theorem G3}, is $F[X]_{S_{v}}$ invertible and hence
principal by Theorem \ref{Theorem G4}. But then $%
F[X]_{S_{v}}=(F[X]_{S_{v}})_{t}=I[X]_{S_{v}}.$
\end{proof}

\begin{lemma}
\label{Lemma G6} Let $I$ be a $t$-ideal of finite type of $D.$ Then $%
I[X]_{S_{v}}$ is $d$-homogeneous if and only if $I$ is $t$-homogeneous.
Consequently $I[X]_{S_{v}}$ is $t$-f-rigid if and only if $I$ is $t$-super
homogeneous.
\end{lemma}

\begin{proof}
Let $I$ be a $t$-homogeneous ideal of $D.$ That $I[X]_{S_{v}}$ is a $t$%
-ideal of finite type is an immediate consequence of Proposition \ref%
{Proposition G2}. If $M$ is the unique maximal $t$-ideal containing $I,$
then at least $M[X]_{S_{v}}\supseteq I[X]_{S_{v}}.$ Suppose that $\mathcal{N}
$ is another maximal ideal of $D[X]_{S_{v}}$ containing $I[X]_{S_{v}}.$ But
by Proposition \ref{Proposition G1}, $\mathcal{N}$ $=N[X]_{S_{v}}$ for some
maximal $t$-ideal $N$ of $D.$ But then $N=D\cap N[X]_{S_{v}}\supseteq D\cap
I[X]_{S_{v}}\supseteq I$ . This forces $N=M$ and consequently $%
N[X]_{S_{v}}=M[X]_{S_{v}}$ making $I[X]_{S_{v}}$ homogeneous.

Conversely if $I[X]_{S_{v}}$ is $d$-homogeneous contained in a unique $%
M[X]_{S_{v}},$suppose that $N$ is another maximal $t$-ideal containing $I.$
Then again $N[X]_{S_{v}}\supseteq ID[X]_{S_{v}}$ which is $d$-homogeneous, a
contradiction unless $N=M.$

The consequently part follows from Corollary \ref{Corollary G5}.
\end{proof}

Let's all a domain $\ast $-f-r-potent if every maximal $\ast $-ideal of $D$
contains a $\ast $-f-rigid element.

\begin{proposition}
\label{Proposition G7} Let $D$ be an integral domain with quotient field $K$%
, $X$ an indeterminate over $D$ and let $S_{v}=\{f\in D[X]|(A_{f})_{v}=D\}.$
Then (a) $D$ is $t$-potent if and only if $D[X]_{S_{v}}$ is $d$-potent and
(b) $D$ is $t$-super potent if and only if $D[X]_{S_{v}}$ is $d$-f-r-potent
\end{proposition}

\begin{proof}
(a) Suppose that $D$ is potent Let $M[X]_{S_{v}}$ be a maximal ideal of $%
D[X]_{S_{v}}$ and let $I$ be a $t$-homogeneous ideal contained in $M.$ By
Lemma \ref{Lemma G6}, $I[X]_{S_{v}}$ is $d$-homogeneous, making $%
M[X]_{S_{v}} $ $d$-potent. Conversely suppose that $D[X]_{S_{v}}$ is $d$%
-potent and let $M $ be a maximal $t$-ideal of $D.$ Then $M[X]_{S_{v}}$ is a
maximal ideal of $D[X]_{S_{v}}$ and so contains a $d$-homogeneous ideal $%
\mathcal{I}$ $=(f_{1},f_{2},...,f_{n})D[X]_{S_{v}}.$ Now let $%
I=(f_{1},f_{2},...,f_{n})$. Then $\mathcal{I}$ $=ID[X]_{S_{v}}$ and $%
I\subseteq (A_{I})_{t}[X]_{S_{v}}\subseteq M[X]_{S_{v}},$ since $%
M[X]_{S_{v}} $ is a $t$-ideal and $f_{i}\in M[X]_{S_{v}}\cap D[X].$ This
gives $\mathcal{I}$ $=ID[X]_{S_{v}}$ $\subseteq
(A_{I})_{t}[X]_{S_{v}}\subseteq M[X]_{S_{v}}$ making $(A_{I})_{t}[X]_{S_{v}}$
another homogeneous ideal, contained in $M[X]_{S_{v}}$ and containing $%
\mathcal{I}$. But then $(A_{I})_{t}$ $\subseteq M$ is a $t$-homogeneous
ideal, by Lemma \ref{Lemma G6}. (b) Use part (a) and Corollary \ref%
{Corollary G5}.
\end{proof}

The other property that can be mentioned \textquotedblleft off
hand\textquotedblright\ is given in the following statement.

\begin{theorem}
\label{Theorem H1} A $t$-f-potent domain of $t$-dimension one is a GCD
domain of finite $t$-character.{}
\end{theorem}

A domain of $t$-dimension one that is of finite $t$-character is called a
weakly Krull domain. ($D$ is weakly Krull if $D=\cap D_{P}$ where $P$ ranges
over a family $\mathcal{F}$ of height one prime ideals of $D$ and each
nonzero non unit of $D$ belongs to at most a finite number of members of $%
\mathcal{F}$.) A weakly Krull domain $D$ is dubbed in \cite{AZ 2019} as $%
\ast $-weakly Krull domain or as a type $1$ $\ast $-SH domain. Here a $\ast $%
-homogeneous ideal $I$ is said to be of type $1$ if $M(I)=\sqrt{I^{\ast }}$
and $D$ is a type $1$ $\ast $-SH domain if every nonzero non unit of $D$ is
a $\ast $-product of finitely many $\ast $-homogeneous ideals of type $1$.
In the following lemma we set $\ast =t.$

\begin{lemma}
\label{Lemma K} A $t$-f-potent weakly Krull domain is a type $1$ $t$-f-SH
domain.
\end{lemma}

\begin{proof}
A weakly Krull domain is a type $1$ $t$-SH domain. But then for every pair $%
I,J$ of similar homogeneous ideals $I^{n}\subseteq J^{\ast }$ and $%
J^{m}\subseteq I^{\ast }$ for some positive integers $m,n.$ So $J$ is a $t$%
-f-homogeneous ideal if $I$ is and vice versa. Thus in a $t$-f-potent weakly
Krull domain the $t$-image of every $t$-homogeneous ideal is principal
whence every nonzero non unit of $D$ is expressible as a product of $t$%
-f-homogeneous elements which makes $D$ a $t$-f-SH domain and hence a GCD
domain.
\end{proof}

\begin{proof}
of Theorem \ref{Theorem H1} Use Theorem 5.3 of \cite{HZ 2019} for $\ast =t$
to decide that $D$ is of finite $t$-character and of $t$-dimension one.
Indeed, that makes $D$ a weakly Krull domain that is $t$-f-potent. The proof
would be complete once we apply Lemma \ref{Lemma K} and note that a $t$-f-SH
domain is a GCD domain and of course of finite $t$-character.
\end{proof}

Generally a domain that is $t$-f-potent and with $t$-dimension $>1,$ is not
necessarily GCD nor of finite $t$-character.

\begin{example}
\label{Example L}$D=Z+XL[[X]]$ where $Z$ is the ring of integers and $L$ is
a proper extension of $Q$ the ring of rational numbers. Indeed $D$ is prime
potent and two dimensional but neither of finite $t$-character nor a GCD
domain.
\end{example}

There are some special cases, in which a $t$-f-potent domain is GCD of
finite $t$-character.

i) If every nonzero prime ideal contains a $t$-f- homogeneous ideal. (Use
(4) of Theorem 5 of \cite{AZ 2019}) along with the fact that $D$ is a $t$%
-f-SH domain if and only if $D$ is a $t$-SH domain with every $t$%
-homogeneous ideal $t$-f-homogeneous. Thus a $t$-f-potent domain of $t$-dim $%
1$ is of finite character.

ii) If $D$ is a $t$-f-potent PVMD of finite $t$-character that contains a
set $S$ multiplicatively generated by $t$-f-homogeneous elements of $D$ and
if $D_{S}$ is a GCD domain then so is $D.$

I'd be doing a grave injustice if I don't mention the fact that before there
was any modern day multiplicative ideal theory there were prime potent
domains as $Z$ the ring of integers and the rings of polynomials over them.
It is also worth mentioning that there are three dimensional prime potent
Prufer domains that are not Bezout. The examples that I have in mind are due
to Loper \cite{Lop 1999}. These are non-Bezout Prufer domains whose maximal
ideals are generated by principal primes.

\section{\label{S3}$\ast $-finite ideal monoids}

In \cite{YZ 2011}, we called a directed p.o. group $G$ pre-Riesz if its
positive elements satisfied the following property.

(pR): If $x_{1},x_{2},...,x_{n}$ are strictly positive elements in $G$ and $%
x_{i}$ are such that there is at least one $g\in G$ with $g\nleq 0$ such
that $g\leq x_{1},x_{2},...,x_{n}$ then there is at least one $r\in G$ such
that $0<r\leq x_{1},x_{2},...,x_{n}.$

By a basic element, in the above paper, we meant a strictly positive element 
$c\in G$ such that for every pair of strictly positive elements $c_{1},c_{2}$
preceding $c$ we have $r\in G$ such that $0<r\leq c_{1},c_{2}.$

Note that it is essentially the positive cone $G^{+}=\{g\in G|g\geq 0\}$ of
the pre-Riesz group that satisfies the (pR), but with reference to elements
of its main group. So let's call a commutative p.o. monoid $M=<M,+,0,\leq ~>$
a pre-Riesz monoid if $M$ is upper directed and satisfies (pR'): For any
finite set of strictly positive elements $x_{1},x_{2},...,x_{n}\in M,$
either $glb(x_{1},x_{2},...,x_{n})=0$ or there is $r\in M$ such that $%
0<r\leq x_{1},x_{2},...,x_{n}.$ Note that the `$+$' and `$0$' are mainly
symbolic, standing in for the monoid operation and the identity. Note also
that to avoid getting into trivialities we shall only consider non-trivial
pre-Riesz monoids, i.e., ones that are different from $\{0\}.$

Here, of course, we do not require that $a\leq b$ $\Leftrightarrow a+x=b$ $.$
The partial order may be pre-assigned but must be compatible with the binary
operation of the monoid. Let's call $M$ a divisibility p.o. monoid if in $M$ 
$a\leq b$ $\Leftrightarrow a+x=b,$ for some $x\in M.$

A monoid $M$ is said to have cancellation if $a+b=a+c$ implies $b=c.$
Obviously if in a cancellation monoid with order defined as above we have $%
a+b\leq a+c$ then $b\leq c.$

\begin{proposition}
\label{Proposition FX} Let $a,b\in M$ where $M$ is a divisibility pre-Riesz
monoid with cancellation. Then $lub(a,b)=a+b$ if and only if $glb(a,b)=0.$
\end{proposition}

\begin{proof}
Suppose that $lub(a,b)=a+b$ and let there be $r>0$ such that $r\leq a,b.$
Then $a=r+x$ and $b=r+y$ for some $x,y\in M.$ Obviously, as $r>0,$ $x<a$ and 
$y<b.$ Thus $r+x+y<a+b$ yet $r+x+y\geq a,b$ contradicting the assumption
that $lub(a,b)=a+b$. Conversely suppose that $glb(a,b)=0$ and let there be,
by way of contradiction, $r$ such that $r\geq a,b$ yet $r<a+b$. Then $%
r=a+x=b+y$ and $a+b=r+z.$ Taking $r=a+x$ we have $a+b=a+x+z.$ Cancelling $a$
from both sides we get $b=x+z.$ Similarly substituting for $r=b+y$ and
cancelling $b$ from both sides we get $a=y+z.$ But then $z\leq a,b$ and
hence $z=0$ forcing $a=y$ and $b=x$ and $r=a+b$, a contradiction.
\end{proof}

If $glb(a,b)$ (resp., $lub(a,b))$ exists in a monoid $M$ we denote it by $%
a\wedge b$ (resp., $a\vee b)$

\begin{example}
\label{Example GX} (1) If $G$ is a Riesz group then as shown in \cite[%
Proposition 3.1]{YZ 2011}, $G^{+}$ is a pre-Riesz monoid. (2) Indeed $G$ is
a pre-Riesz group if and only if $G^{+}$ is a pre-Riesz monoid and indeed a
pre-Riesz group can be regarded as a pre-Riesz monoid. (3) Let $\ast $ be a
finite character star operation defined on a domain $D$ and let $\Gamma $ be
the set of all proper $\ast $-ideals of finite type of $D$. Then $\Gamma
\cup \{D\}$ is a pre-Riesz monoid under $\ast $-multiplication because $%
glb(A_{1},A_{2},...,A_{n})=D$ if and only if $(A_{1},A_{2},...,A_{n})^{\ast
}=D.$ Let's denote this monoid by $<\Gamma \cup \{D\},\ast ,D,\leq $ $>$ and
call it $\ast $-finite ideals monoid ($\ast $-FIM)
\end{example}

(This is because the $\ast $-product of finitely many members of $\Gamma $
is again of finite type and this $\ast $-product is associative. Here the
partial order is induced by reverse containment i.e. for $A,B\in \Gamma ,$ $%
A\leq B$ if and only if $A\supseteq B$ and of course $\ast $-multiplication
is compatible with the order, i.e., for $A,B,C\in \Gamma $ with $A\leq B$
then $(AC)^{\ast }\leq (BC)^{\ast }$ (since $A\supseteq B$ implies that $%
(AC)^{\ast }\supseteq (BC)^{\ast }).)$

Let's call $D$ $\ast $-coherent if for all $A,B\in \Gamma $ we have $A\cap
B\in \Gamma .$

\begin{proposition}
\label{Proposition HX} Let $<\Gamma \cup \{D\},\ast ,D,\leq $ $>$ be a $\ast 
$-finite monoid (1) For all $H,K\in \Gamma ,$ (We have $H\wedge K\in \Gamma $
as $(H,K)^{\ast }$ and if $H\cap K\in \Gamma ,$ then $H\vee K=H\cap K.$
\end{proposition}

\begin{proof}
Indeed as $(H,K)^{\ast }\leq H,K$ (sine $(H,K)^{\ast }\supseteq H,K)$ and if 
$A\leq H,K$ for some $A\in \Gamma ,$ (i.e., $A\supseteq H,K$) then $A\leq
(H,K)^{\ast }.$ Let's put it this way $(H,K)^{\ast }$ is standard for $\inf
(H,K)$ and $H\cap K,$ if it exists is standard for $\sup (H,K)$ in ideal
theory and so it is here.
\end{proof}

So, a $\ast $-finite monoid is actually a semilattice. Now let $M$ be a
pre-Riesz monoid and $H\in M.$ Call $H$ homogeneous if for all $0<R,S\leq H$
we have a $0<t<R,S.$ Obviously $0<K\in M$ is not homogeneous if and only if
there are $0<R,S<K$ such that $\inf (R,S)=0.$ Let's call $X,Y\in M$ disjoint
if $\inf (X,Y)=0$ and note that if $H$ is homogeneous then $H$ cannot be non
disjoint with two or more disjoint elements. Also if $X,Y$ are disjoint and $%
0\,<x<X$ then $x$ and $Y$ are disjoint, for if not then there is $0<r<x,Y,X$
making $X,Y$ non-disjoint.

Call a set $S$ of homogeneous elements of a pre-Riesz monoid $M$ an
independent set if every pair of elements of $S$ is disjoint. In notes of my
work with Yang and a student of his \cite{LYZ 2014}, other, restricted,
versions of the following were proved. As the notes are not made public yet
and there is a significant difference of the notions involved, I include
below some related results that are relevant to this write up.

\begin{proposition}
\label{Proposition JX} Let $S$ be an independent set of homogeneous
elements, in a pre-Riesz monoid, satisfying a property $(Q).$ Then $S$ can
be enlarged to a maximal independent set $T$ of homogeneous elements
satisfying $(Q).$
\end{proposition}

\begin{proof}
Let $\Gamma =\{B|B\supseteq S$ is an independent set of homogeneous elements
satisfying $(Q)\}$. Obviously $\Gamma \neq \phi $. Now let $\{B_{\alpha }\}$
be a chain of members of $\Gamma $ and let $C=\cup B_{\alpha }.$ Then $%
C\supseteq S$ and for any pair $x,y\in C$, $x,y$ are in $B_{\alpha }$ for
some $\alpha $ so elements of $C$ are homogeneous, satisfy $(Q)$ and are
homogeneous. So, $C\in \Gamma .$ Thus by Zorn's Lemma $\Gamma $ must contain
a maximal element and that is our $T.$
\end{proof}

We shall call a set $S$ of mutually disjoint elements, of a monoid $M,$ a
maximal disjoint set if (as usual) no set $T$ exists of mutually disjoint
elements such that $M\supseteq T\supsetneq S$ and we shall call a set $S$ of
mutually disjoint elements of $M$ order maximal if no element $s$ of $S$ can
be replaced by two distinct predecessors to form a set $(S\backslash
\{s\})\cup \{x,y)$ of mutually disjoint elements. A maximal set of disjoint
homogeneous elements is obviously order maximal too, but a mere maximal set
of mutually disjoint elements may not be, as we have seen in the case of
ideals in a ring.

An order maximal independent set $B$ of homogeneous elements of a pre-Riesz
monoid $M$ is called a basis if $B$ is also an order maximal set of mutually
disjoint elements.

\begin{lemma}
\label{Lemma KX}. Let $M$ be a pre-Riesz monoid. Then a non-empty subset $S$
of $M$ is a basis if and only if $S$ is disjoint and ($S\backslash \{s\}$)$%
\cup \{x,y\}$ is non-disjoint for any $s\in S$ and for any $\{x,y\}\subseteq 
$ $(M\backslash S)\cup \{s\}$, with $x\neq y$.
\end{lemma}

\begin{proof}
Let $S$ be a basis and suppose that for some $s\in S$, $(S\backslash
\{s\})\cup \{x,y\}$ is disjoint for some $\{x,y\}\subseteq $ $(M\backslash
S)\cup \{s\}$, with $x\neq y$. Then $x\wedge s\neq 0$ and $y\wedge s\neq 0$
because $S$ is a maximal set of disjoint elements of $M$. Since $M$ is
pre-Riesz, there is $t\in M$ such that $0<t\leq x,s$ and $u\in M$ such that $%
0<u\leq y,s$. Next as $s$ is homogeneous, there is $w\in M$ such that $%
0<w\leq t,u,x,y$, a contradiction. Conversely, suppose that $S$ is disjoint
and satisfies the condition in the lemma. If $S$ $\cup \{x\}$ is disjoint
for some $x\in M\backslash S$, then for any $s\in S$, $(S\backslash
\{s\})\cup \{s,x\}$ is disjoint and $s\neq x$, a contradiction. Therefore, $%
S $ is an maximal disjoint set. If $s\in S$ and $s$ is not homogeneous, then
there exists at least one pair of elements $0<x,y<s$ such that $x\wedge y=0$%
. But then $x,y\notin S,x\neq y$ and $(S\backslash \{s\})\cup \{x,y\}$ is
disjoint, a contradiction. Thus, $S$ is a maximal disjoint set consisting of
homogeneous elements, i.e., $S$ is a basis.
\end{proof}

\begin{theorem}
\label{Theorem LX} (cf \cite[Theorem 9]{LYZ 2014}) A pre-Riesz monoid $M$
has a basis if and only if (P): each $0<x\in M$ exceeds at least one
homogeneous element . Every basis of $M$ is an order maximal independent
subset and every order maximal independent subset of $M$ is a basis provided 
$M$ has a basis.
\end{theorem}

\begin{proof}
Let $S$ = $\{0<a_{\gamma }|\gamma \in \Gamma \}$ be a basis for $M$, and
consider $0<x\in M$. There exists $\gamma \in \Gamma $ such that $x\wedge
a_{\gamma }\neq 0$ for otherwise $S$ is not a maximal set of disjoint
elements. This means that there is $0<h\leq x,a_{\gamma }$ and $h$ is
homogeneous because $a_{\gamma }$ is homogeneous and $h\in (0,a_{\gamma }]$.
Thus, $M$ satisfies (P). Conversely, suppose that $M$ satisfies the property
(P). Since $M$ is non-trivial there is at least one homogeneous element and,
by Proposition \ref{Proposition JX}, there exists a maximal independent
subset $T=\{0<a_{\gamma }|\gamma \in \Gamma \}$ of $M$, assuming that $(Q)$
means \textquotedblleft no restriction". All we need show is that $T$ is a
maximal set of disjoint elements. Suppose on the contrary that there is an
element $0<x\in M\backslash T$ such that $x\wedge a_{\gamma }=0$ for all $%
\gamma \in \Gamma $. But then by the property (P), $x$ exceeds a homogeneous
element $h$, and $h$ is disjoint with $a_{\gamma }$ for all $\gamma \in
\Gamma .$ Therefore, $T\cup \{h\}\supsetneq T$and $T$ $\cup \{h\}$ is an
independent subset of $M$, but this is contrary to our choice of $T$.
\end{proof}

Conrad's F-condition on a pre-Riesz monoid reads thus: Each strictly
positive element $x$ in a pre-Riesz monoid $M$ is greater than at most a
finite number of (mutually) disjoint positive elements.

\begin{proposition}
\label{Proposition MX} If a pre-Riesz monoid $M$ satisfies Conrad's
F-condition, then $M$ has a basis.
\end{proposition}

\begin{proof}
Suppose that the condition holds but $M$ has no basis. Then by Theorem \ref%
{Theorem LX}, there is at least one $0<y\in M$ such that no homogeneous
element is contained in $\{x\in M:0<x\leq y\}$. Then there exist two
disjoint elements $x_{1},y_{1}$ with $0<x_{1},y_{1}<y$ where none of $%
x_{1},y_{1}$ exceeds a homogeneous element for otherwise $y$ would. So, say, 
$0<x_{2},y_{2}<x_{1}$ with $x_{2}\wedge y_{2}$ $=0$. Since $x_{1}\wedge
y_{1}=0$ and $y_{2}<x_{1}$ we have $y_{1}\wedge y_{2}=0$. Next $%
0<x_{3},y_{3}<x_{2}$ with $x_{3}\wedge y_{3}=0$. We can conclude that $%
y_{1},y_{2},y_{3}$ are mutually disjoint. Similarly producing $x_{i}$ s, $%
y_{i}$ s and using induction we can produce an infinite sequence $\{y_{i}\}$
of mutually disjoint elements less than $y$. Contradicting the assumption
that $M$ satisfies Conrad's F-condition.
\end{proof}

\begin{corollary}
\label{Corollary NX} The following are equivalent for a pre-Riesz monoid $M$%
: (i) $M$ satisfies Conrad's F-condition, (ii) Every strictly positive
element exceeds at least one and at most a finite number of homogeneous
elements that are mutually disjoint, (iii) $M$ contains a subset $\Gamma $
of strictly positive elements such that every strictly positive element of $%
M $ exceeds at least one member of $\Gamma $ and at most a finite number of
mutually disjoint members of $\Gamma .$
\end{corollary}

\begin{proof}
(i) $\Rightarrow $ (ii) Conrad's F-condition, via Proposition \ref%
{Proposition MX}, implies that every strictly positive element $x$ exceeds
at least one homogeneous element say $h.$ The set $\{h\}$ is an independent
set of ($(Q)$)homogeneous elements preceding $x$ and by Proposition \ref%
{Proposition JX}, $\{h\}$ can be expanded to a maximal independent set $T$
of elements preceding $x.$ But again by Conrad's F-condition, $T$ must be
finite. For (ii) )$\Rightarrow $(i), suppose that (ii) holds yet $M$ does
not satisfy (i). Then there is $0<x\in M$ that exceeds an infinite sequence $%
\{x_{i}\}$ of mutually disjoint strictly positive elements of $M$. Now each
of $x_{i}$ exceeds at least one homogeneous element $h_{i}$. Since $%
\{x_{i}\} $ are mutually disjoint, $\{h_{i}\}$ are mutually disjoint, which
causes a contradiction. Whence, we have the conclusion. (ii)$\Rightarrow $
(iii) Take $\Gamma =\{h|h$ is a homogeneous element of $M\}$, then every
positive element exceeds at least one member of $\Gamma $ and at most a
finite number. (iii)$\Rightarrow $ (i) Suppose that the given condition
holds but Conrad's F-condition doesn't. That means there is some element $%
y>0 $ such that $y$ is greater than an infinite number of mutually disjoint
elements $\{y_{\alpha }\}$ of $M.$ By (iii) each $y_{\alpha }$ exceeds a
member $z_{\alpha }$ of $\Gamma .$ As $y_{\alpha }$ are mutually disjoint,
making $y$ exceed an infinite number of mutually disjoint members of $\Gamma
,$ a contradiction.
\end{proof}

\begin{corollary}
\label{Corollary PX} (Corollary to Corollary \ref{Corollary NX}) Let $D$ be
an integral domain, $\ast $ a finite character star operation on $D$ and let 
$\Gamma $ be a set of proper, nonzero, $\ast $-ideals of finite type of $D$
such that every proper nonzero finite type $\ast $-ideal of $D$ is contained
in some member of $\Gamma $ . Then $D$ is of finite $\ast $-character if and
only if every nonzero finitely generated ideal $I$ of $D$ with $I^{\ast
}\neq D$ is contained in at least one and at most a finite number of
mutually $\ast $-comaximal members of $\Gamma .$
\end{corollary}

\begin{proof}
We know that $M=\{A^{\ast }|$ $A^{\ast }\neq D$ is a $\ast $-ideal of finite
type of $D\}\cup \{D\}$ is a pre-Riesz monoid under $\ast $-multiplication
and the set $\Gamma $ can just be regarded as a subset of $M$ and the
theorem requires every strictly positive member of $M$ exceeds at least one
member of $\Gamma $ and at most a finite number of mutually disjoint members
of $\Gamma .$ Now this means, according to Corollary \ref{Corollary NX},
that every element $A^{\ast }$ exceeds at least one basic element and at
most a finite number of basic elements of $M.$ Now take an element $A^{\ast
} $ in $M$ and let $h$ be a basic element of $\Gamma $ containing $A^{\ast
}. $ Then, by Proposition \ref{Proposition JX}, there is at least one
maximal set $S$ of mutually disjoint basic elements containing $A^{\ast }$
and each $h\in S$ exceeds some member of $\Gamma $ giving a maximal set $T$
of basic elements in $\Gamma $ and containing $A^{\ast }$ Now this
translates to: If the condition is satisfied, then or every $\ast $-ideal of
finite type $A$ there is a maximal set $T$ of homogeneous $\ast $-ideals
containing $A$ and by the condition, $T$ is finite. Now let $|T|=n$ and
recall that if $T=\{H_{1},...,H_{n}\}$ then each of the $H_{i}$ determines a
unique maximal $\ast $-ideal $M(H_{i}).$ To show that $T^{\prime
}=\{M(H_{1}),...,M(H_{n})\}$ contains all the maximal $\ast $-ideals
containing $A^{\ast }$ assume that there is a maximal $\ast $-ideal $N\notin
T^{\prime }$ and containing $A^{\ast }.$ Then there is $x\in N\backslash
(\cup M(H_{i}).$ But then $xD$ is $\ast $-comaximal with $H_{i}$ for each $i$
and hence $(x,A)^{\ast }\subseteq N$ is $\ast $-comaximal with each $H_{i}$
which translates to: $(x,A)^{\ast }$ is disjoint with each basic element $%
H_{i}.$ But then $(x,A)^{\ast }$ exceeds a basic element $K$ which must be
disjoint with each of $H_{i}$, killing the maximality of $T.$ The converse
is obvious because if there is an infinite number of mutually $\ast $%
-comaximal members of $\Gamma $ then $D$ cannot be of finite $\ast $%
-character because a maximal $\ast $-ideal cannot afford mutually $\ast $%
-comaximal ideals.
\end{proof}

Finally, it's important to mention that not all p.o. monoids are pre-Riesz
monoids. According to Proposition 4.2 of \cite{YZ 2011} The group of
divisibility $G(D)$ of a domain $D$ is pre-Riesz if and only if (P): for all 
$x_{1},x_{2},...,x_{n}\in D\backslash \{0\},$ $(x_{1},x_{2},...,x_{n})_{v}=D$
or $(x_{1},x_{2},...,x_{n})\subseteq rD$ for some non unit $r\in D.$ As we
can readily see, a domain satisfying (P) above is a domain satisfying the
PSP property and in a PSP domain every atom is a prime. Thus an atomic
domain (every nonzero non unit is expressible as a product of atoms) with
PSP property is a UFD. Thus, say, if $D$ is a non UFD Noetherian domain then 
$G(D)$ is not pre-Riesz. It may be noted that the set of principal ideals is
under multiplication is a submonoid of $\Gamma \cup \{D\}.$

\subsection{Riesz monoids}

First off let's note that when we say "monoid" we mean a commutative monoid.
Now call a directed p.o. monoid $M=<M,+,0,\leq $ $>$ a sub-Riesz monoid, if
every element $x$ of $M$ is primal i.e. for $y_{1},y_{2}\in M,$ $x\leq
y_{1}+y_{2}$ $\Rightarrow x=x_{1}+x_{2}$ such that $x_{i}\leq y_{i}$ and a
Riesz monoid if $M$ is also divisibility and cancellative.

One may ask whether Riesz monoids satisfy the Riesz interpolation, as do
Riesz groups. The answer is yes and can be readily checked as we show below.
Note that by $M^{+}$ we mean the set $\{x\in M|x\geq 0\}$

\begin{theorem}
\bigskip \label{Theorem QX}TFAE for a commutative cancellation divisibility
monoid $M$. (1) Every $0\leq x\in M$ is primal (2) For all $a,b,x,y\in M^{+}$
with $a,b\leq $ $x,y$ there is $z$ such that $a,b\leq z\leq x,y,$ (3) For
all $a,b,x_{1},x_{2},...,x_{n}$ $\in M^{+}$ with $a,b\leq $ $%
x_{1},x_{2},...,x_{n}$ there exists $z$ such that $a,b\leq z\leq
x_{1},x_{2},...,x_{n}$, (4) For all $%
a_{1},a_{2},...,a_{n},b_{1},b_{2},...,b_{m}\in M^{+}$ with $%
a_{1},a_{2},...,a_{n}\leq $ $b_{1},b_{2},...,b_{m}$ there exists $d$ such
that $a_{1},a_{2},...,a_{n}\leq d\leq b_{1},b_{2},...,b_{m}.$
\end{theorem}

\begin{proof}
(1) $\Rightarrow $ (2) Let every positive element of $M$ be primal.

Let $a,b\leq x,y.$ Then $x=x_{1}+a=x_{2}+b$ and $y=y_{1}+a=y_{2}+b......(1)$

Since $x_{1}+a=x_{2}+b,$ $b\leq x_{1}+a$ and since $b$ is primal $%
b=b_{1}+b_{2}$ where $b_{1}\leq x_{1}$ and $b_{2}\leq a.$ (2)

Let $x_{1}=x_{1}^{\prime }+b_{1}$ and $a=a_{1}+b_{2}.$ Then $x_{1}+a=x_{2}+b$
can be written as $x_{1}^{\prime }+b_{1}+a_{1}+b_{2}=x_{2}+b$, or $%
x_{1}^{\prime }+a_{1}+b_{1}+b_{2}=x_{2}+b.$ Noting that $b=b_{1}+b_{2}$ and
cancelling $b$ from both sides we get $x_{1}^{\prime }+a_{1}=x_{2}.$
......(3)

Since $a_{1}+b_{2}=a$ we have $a,b\leq a_{1}+b........(4)$

Using the value of $x_{2}$ we have $a_{1}+b\leq x.$ (Note: $%
x=x_{2}+b=(x_{1}^{\prime }+a_{1})+b)$ ... (5)

Now consider $y_{1}+a=y_{2}+b.$ Using $a=a_{1}+b_{2}$ and $b=b_{1}+b_{2}$ we
have $y_{1}+a_{1}+b_{2}=y_{2}+b_{1}+b_{2}.$ Cancelling $b_{2}$ from both
sides we get $y_{1}+a_{1}=y_{2}+b_{1}.$ So that $b_{1}\leq y_{1}+a_{1}$ and
as $b_{1}$ is primal we have $b_{1}=b_{3}+b_{4}$ where $b_{3}\leq y_{1}$ and 
$b_{4}\leq a_{1}.$ Writing $y_{1}=y_{1}^{\prime }+b_{3}$ and $%
a_{1}=a_{1}^{\prime }+b_{4}$ we can express $y_{1}+a_{1}=y_{2}+b_{1}$ as $%
y_{1}^{\prime }+b_{3}+a_{1}^{\prime }+b_{4}=y_{2}+b_{1}.$ Cancelling $%
b_{1}=b_{3}+b_{4}$ from both sides we get $y_{2}=y_{1}^{\prime
}+a_{1}^{\prime }.$ This gives $y=y_{2}+b=y_{1}^{\prime }+a_{1}^{\prime
}+b=y_{1}+a.$ Now as $y_{1}^{\prime }\leq y_{1}$ we get $y_{1}=y_{4}+y_{1}^{%
\prime }$ which on substituting in $y_{1}^{\prime }+a_{1}^{\prime
}+b=y_{1}+a $ gives $y_{1}^{\prime }+a_{1}^{\prime }+b=y_{4}+y_{1}^{\prime
}+a$ and cancelling $y_{1}^{\prime }$ we get $y_{4}+a=a_{1}^{\prime }+b$ and
so $a\leq a_{1}^{\prime }+b.$ That is $a,b\leq a_{1}^{\prime }+b$ and $%
a_{1}^{\prime }+b\leq y.$ But as $a_{1}^{\prime }\leq a_{1}$ and $%
x_{2}=x_{1}^{\prime }+a_{1}$ we have $a_{1}^{\prime }+b\leq x_{2}+b=x.$ So
we have $z=a_{1}^{\prime }+b$ such that $a,b\leq z\leq x,y.$

(2) $\Rightarrow $ (1). Let $a\leq b+c$.

Then as $a,b\leq b+c,$ $a+b$ there is $x$ such that $a,b\leq x\leq b+c,$ $%
a+b $ ..........(i)

Now as $a\leq x$ we have $x=x_{1}+a$ ..........(ii)

and as $b\leq x$ we have $x=x_{2}+b.......$...(iii)

Using (i) and (iii) $x_{2}\leq a$ and $x_{2}\leq c.$ Now as $x_{2}\leq a,$
setting $a=x_{3}+x_{2}$ we have from $x_{1}+a=x_{2}+b$, the equation $%
b=x_{1}+x_{3}.$ So $a\leq b+c$ implies that $a=x_{2}a+x_{3},$ with $%
x_{2},x_{3}\in M^{+}$ such that $x_{3}\leq b$ and $x_{2}\leq c.$

(2) $\Rightarrow $ (3). Let $a,b\leq x_{1},x_{2},...,x_{n}.$ If $n=2$ we
have the result by (2). So suppose that $n>2$ and suppose that for all $%
x_{1},x_{2},...,x_{n-1}$ the statement is true. Then for $a,b\leq
x_{1},x_{2},...,x_{n-1}$ there is a $d_{1}$ such that $a,b\leq d_{1}\leq
x_{1},x_{2},...,x_{n-1}.$ But then for $d_{1},x_{n}$ there is $d$ with $%
a,b\leq d\leq d_{1},x_{n}.$ But this $d$ satisfies $a,b\leq d\leq
x_{1},x_{2},...,x_{n}.$

(3) $\Rightarrow $ (4). Let $a_{1},a_{2},...,a_{n}\leq
b_{1},b_{2},...,b_{m}. $ Then $a_{1},a_{2}\leq b_{1},b_{2},...,b_{m}$ and so
there is a $d_{1}$ such that $a_{1},a_{2}\leq d_{1}\leq
b_{1},b_{2},...,b_{m}.$ Now $d_{1},a_{3},...,a_{n}\leq b_{1},b_{2},...,b_{m}$
and induction on $n$ completes the job. (4) $\Rightarrow $ (2). Obvious
because (2) is a special case of (4).
\end{proof}

Part (2) of Theorrem \ref{Theorem QX} is also called $(2,2)$ Riesz
interpolation Property and (4) is $(n,m)$ interpolaion for positive integral 
$n$ and $m$.

Call a subset $S$ of a monoid $M$ conic if $x+y=0$ implies $x=0=y,$ for all $%
x,y\in S.$ In a p.o. group $G$ the sets $G^{+}$ and $-G^{+}$ are conic. If $%
D $ is an integral domain then the set $m(D)$ of nonzero principal ideals of 
$D $ is a monoid under multiplication, with identity $D,$ ordered by $aD\leq
bD$ $\Leftrightarrow $ there is $c\in D$ such that $bD=acD\Leftrightarrow
aD\supseteq bD.$ The monoid $m(D)$ is cancellative too and in $m(D)$ $%
xDyD=1\Rightarrow xD=yD=1.$ So, $m(D)$ is a divisibility cancellative conic
monoid. The monoid $m(D)$ is of interest because of the manner it generates
a group. We know how the field of quotients of a domain is formed as a set
of ordered pairs eah pair representing an equivalence class with $%
(a,b)=(c,d) $ $\Leftrightarrow da=bc$ and then we represent the pair $(a,b),$
$b\in D\backslash \{0\}$ by $\frac{a}{b}=ab^{-1}.$ Now the group of $m(D)$
gets the form $G(D)=\{\frac{a}{b}D|\frac{a}{b}\in qf(D)\backslash \{0\}\},$
ordered by $\frac{a}{b}D\leq \frac{c}{d}D$ $\Leftrightarrow \frac{a}{b}%
D\supseteq \frac{c}{d}D\Leftrightarrow $ there is $hD\in m(D)$ such that $%
\frac{a}{b}DhD=\frac{a}{b}D,$ so that $m(D)$ is the positive cone of $G(D).$
The group $G(D)$ gets the name group of divisibility of $D$ (actually of $%
m(D)).$ Now any divisibility monoid that is also a cancellative and conic
monoid $M,$ with least element $0$ can be put through a similar process of
forming equivalent classes of ordered pairs to get group of divisibility
like group $G(M)=\{a-b|a,b\in M\}$ with $x\leq y$ in $G(M)$ $\Leftrightarrow
x+h=y$ for some $h\in M.$

\begin{corollary}
\label{Corollary RX}A Riesz Monoid $M$ has the pre-Riesz property. Also $%
M^{+}$ is conic for a Riesz monoid $M.$
\end{corollary}

\begin{proof}
Let $0\leq x,y$ in $M$ and suppose that there is $g\in M$ such that $g$ is
not greater than or equal to $0$ yet $g\leq x,y,$ that is $0,g\leq x,y.$
Then by the $(2,2)$ interpolation property there is $r\in M$ such that $%
0,g\leq r\leq x,y.$ But then $r>0,$ as $r\geq 0$ and $r\neq 0$ because $%
r\geq g.$ Next suppose $x,y\geq 0$. If $x+y=0$ and say $x\neq 0,$ then we
have $0,x\leq x,x+y$ and by the $(2,2)$ interpolation there is $r$ such that 
$0<x,x+y$ contradicting the fact that $x+y=0.$
\end{proof}

Well a p.o. monoid $M$ is a p.o. group if every element of $M$ has an
inverse and obviously if a p.o. monoid is a Riesz monoid and a group it is a
Riesz group. This brings up the question: Let $M$ be a Riesz monoid and $%
M^{+}$ the positive cone of it, will $M^{+}$ generate a Riesz group? As we
shall be mostly concerned with monoids $M$ with $0$ the least element, i.e. $%
M=M^{+}$ we remodel the question as: Let $M$ be a Riesz monoid with $M^{+}=M$
the positive cone of it, will $M$ generate a Riesz group? The following
result whose proof was indicated to me by G.M. Bergman, in an email,
provides the answer.

\begin{theorem}
\label{Theorem SX}Suppose $M$ is a cancellative abelian monoid, which is
"conical", i.e., no two nonidentity elements sum to $0$, and which we
partially order by divisibility; and suppose every element of $M$ is primal,
namely, that with respect the divisibility order, (1) $x\leq a+b=>x=u+v$
such that $u\leq a$ and $v\leq b$. Then the group generated by $M$ is a
Riesz group.
\end{theorem}

\begin{proof}
Let us rewrite (1) by translating all the inequalities into their
divisibility statements; so that $x\leq a+b$ becomes $x+y=a+b$ for some $y$
and $u\leq a$ becomes $a=u+u\prime $, and similarly for the last inequality;
and finally, let us rename the elements more systematically; in particular,
using $a,b,c,d$ for the above$x,y,a,b$. Then we find that (1) becomes $%
a+b=c+d\Rightarrow a=a\prime +a",c=a\prime +b\prime ,d=a"+b"$ for some $%
a\prime ,a",b\prime ,b"\in M.$ Now if we substitute the three equations to
the right of the "$\Rightarrow $" into the equation before the "$\Rightarrow 
$", and use cancellativity, we find that $b=b\prime +b"$; so the full
statement is (2) $a+b=c+d\Rightarrow a=a\prime +a",~b=b\prime +b",~c=a\prime
+b\prime ,~d=a"+b",$ for some $a\prime ,a",b\prime ,b"\in M$. Now let $G$ be
the group generated by $M$, ordered so that $M$ is the positive cone. We
want to show $G$ has the Riesz Interpolation property. So suppose that in $G$
we have $p,q\leq r,s$. We can write these inequalities as (3) $%
r=p+a,s=p+c,r=q+d,s=q+b$ where $a,b,c,d\in M$. Now the sum of the first and
last equations gives a formula for $r+s$, and so does the sum of the second
and third equations.Equating the results, and cancelling the summands $p+q$
on each side, we get an equation in $M:a+b=c+d$. Hence we can apply (2) to
get decompositions of $a,b,c,d$, and substitute these into (3), getting (4) $%
r=p+a\prime +a",s=p+a\prime +b\prime ,r=q+a"+b",s=q+b\prime +b"$. Equating
the first and third equations (or if we prefer, the second and fourth) and
cancelling the common term $a"$ (respectively, the common term $b\prime $),
we get (whichever choice we have made) (5) $p+a\prime =q+b"$.The element
given by (5) is clearly $\geq p,q$, while from (4) (using whichever of the
equations for $r$ we prefer and whichever of the equations for $s$ we
prefer), we see that it is $\leq r,s$. So this is the element whose
existence is required for the ($(2,2)$) Riesz interpolation property for $G$.
\end{proof}

A fractional ideal $I$ is called $\ast $-invertible if $(II^{-1})^{\ast }=D.$
It is well known that if $I$ is $\ast $-invertible for a finite character
star operation $\ast $ then $I^{\ast }$ and $I^{-1}$ are of finite type.
Denote the set of all $\ast $-invertible fractional $\ast $-ideals of $D$ by 
$Inv_{\ast }(D)$ and note that given an integral ideal $I$ it is possible
that $I$ cannot always be expressed as a product of \bigskip integral
ideals. So when we talk about an integral $\ast $-invertible $\ast $-ideal
we are talking about the end result and not how it is expressed. Let $%
\mathcal{I}_{\ast }(D)$ be the set of integral $\ast $-invertible $\ast $%
-ideals and note that $\mathcal{I}_{\ast }(D)$ is a monoid under $\ast $%
-multiplication. Note that $\mathcal{I}_{\ast }(D)$ can be partially ordered
by $I\leq J$ if and only if $I\supseteq J$. Indeed $J\subseteq I$ if and
only if $(JI^{-1})^{\ast }=H\subseteq D$, if and only if $J=(IH)^{\ast },$
and as $J,$ $I$ are $\ast $-invertible, $H$ is $\ast $-invertible and
integral. Thus in $\mathcal{I}_{\ast }(D)$, $I\leq J$ $\Leftrightarrow
J=(IH)^{\ast }$ for some $H\in \mathcal{I}_{\ast }(D).$ In other words $%
\mathcal{I}_{\ast }(D)$ is a divisibility p.o. monoid. Because $\mathcal{I}%
_{\ast }(D)$ involves only $\ast $-invertible $\ast $-ideals, it is
cancellative too. Finally $\mathcal{I}_{\ast }(D)$ is directed because of
the definition of order. That $Inv_{\ast }(D)$ is generated by $\mathcal{I}%
_{\ast }(D)$ follows from the fact that every fractionary ideal of $D$ can
be written in the form $A/d$ where $A\in F(D)$ and $d\in D\backslash \{0\}.$
Finally, the partial order in $Inv_{\ast }(D)$ gets induced by $\mathcal{I}%
_{\ast }(D)$ in that for $I,J\in $ $Inv_{\ast }(D)$ we have $I\leq J$ $%
\Leftrightarrow J\subseteq I\Leftrightarrow (JI^{-1})^{\ast }\in $ $\mathcal{%
I}_{\ast }(D).$ Call $I\in \mathcal{I}_{\ast }(D)$ $\ast $-primal if for all 
$J,K\in \mathcal{I}_{\ast }(D)$ $I\leq (JK)^{\ast }$ we have $%
I=(I_{1}I_{2})^{\ast }$ where $I_{1}^{\ast }\leq J$ and $I_{2}^{\ast }\leq
K. $ Call $D$ $\ast $-Schreier, for star operation $\ast $ of finite
character, if every integral $\ast $-invertible $\ast $-ideal of $D$ is $%
\ast $-primal.

\begin{proposition}
\label{Proposition TX}Let $\ast $ be a finite character star operation
defined on $D.$ Then $D$ is a $\ast $-Schreier domain if and only if $%
Inv_{\ast }(D)$ is a Riesz group under $\ast $-multiplication and order
defined by $A\leq B\Leftrightarrow A\supseteq B.$
\end{proposition}

\begin{proof}
Suppose that $D$ is $\ast $-Schreier, as defined above. That is each $I\in $ 
$\mathcal{I}_{\ast }(D)$ is primal. The notion of $\ast $-Schreier suggests
that we define $\leq $ by $A\leq B$ $\Leftrightarrow A\supseteq B.$ Then as
for each pair of integral ideals $IJ$, $(IJ)^{\ast }=D$ $\Rightarrow J^{\ast
}=I^{\ast }=D,$ the same holds for members of $\mathcal{I}_{\ast }(D)$ which
are all $\ast $-ideals. So $(IJ)^{\ast }=D\Rightarrow I=J=D.$ and so $%
\mathcal{I}_{\ast }(D)$ is conic. Of course $\mathcal{I}_{\ast }(D)$ is
cancellative by the choice of ideals and by the definition of ordr $\mathcal{%
I}_{\ast }(D)$ is a divisibility monoid. So by Theorem \ref{Theorem SX} $%
\mathcal{I}_{\ast }(D)$ generates a Riesz group and by the above
considerations $Inv_{\ast }(D)$ is generated by $\mathcal{I}_{\ast }(D).$
Consequently $Inv_{\ast }(D)$ is a Riesz group. Conversely if $Inv_{\ast
}(D) $ is a Riesz group, with that order defined on it, then $\mathcal{I}%
_{\ast }(D)$ is the positive cone of the Riesz group $Inv_{\ast }(D)$ and so
each element of $\mathcal{I}_{\ast }(D)$ must be primal.
\end{proof}

Proposition \ref{Proposition TX} brings together a number of notions studied
at different times. The first was quasi-Schreier, study started in \cite{DM
2003} and completed in \cite{ADZ 2007}. The target in these papers was
studying $\mathcal{I}_{d}(D),$ i.e. the monoid of invertible integral ideals
of $D,$ when $Inv_{\ast }(D)$ is a Riesz group. Another study targeting $%
\mathcal{I}_{t}(D),$ i.e. the monoid of $t$-invertible integral $t$-ideals
of $D,$ for study along the same lines as above appeared in \cite{DZ 2011}.

Now let's step back and require that every $\ast $-invertible $\ast $-ideal
of $D$ be principal. Then in Proposition \ref{Proposition TX}, $\mathcal{I}%
_{\ast }(D)$ is the monoid of principal ideals, each of which is primal and
the Riesz group $Inv_{\ast }(D)$ of consists just of principal fractional
ideals of $D$, and hence the group of divisibility of $D$. It is well known
that if $\ast $ is of finite type each $\ast $-invertible $\ast $-ideal is a 
$t$-invertible $t$-ideal (\cite{Zaf 2000}) and that in a pre-Schreier domain
each $t$-invertible $t$-ideal is principal (\cite[Theorem 3.6]{Zaf 1987}).
So we have the following corollary.

\begin{corollary}
\label{Corollary VX}Let $D$ be $\ast $-Schreier for any star operation $\ast 
$ of finite character. Then $D$ is pre-Schreier if and only if each element
of $\mathcal{I}_{\ast }(D)$ is principal.
\end{corollary}

\begin{proof}
Suppose that eah member of $\mathcal{I}_{\ast }(D)$ is principal then in $%
\mathcal{I}_{\ast }(D)$. Then for $a,b,c\in D\backslash \{0\}$ we have $%
aD,bD,cD\in \mathcal{I}_{\ast }(D)$ and for $a|bc$ in $D$ would be $aD\leq
bDcD$ and in $\mathcal{I}_{\ast }(D)$ we must have $aD=(I_{1}I_{2})^{\ast }$
where $I_{1}\leq bD$ and $I_{2}\leq cD.$ But $I_{i}$ being in $\mathcal{I}%
_{\ast }(D)$ must be principal. So,say, $I_{i}=a_{i}D$. But this gives $%
a=a_{1}a_{2}$ and $a_{1}D\leq bD,a_{2}D\leq cD$ gives $a_{1}|b,a_{2}|c.$ In
sum for all $a,b,c\in D\backslash \{0\}$ $a|bc$ $\Rightarrow a=a_{1}a_{2}$
where $a_{1}|b$ and $a_{2}|c$ which is a way of saying that every nonzero
element of $D$ is primal. Conversely as indicated earlier $D$ being
pre-Schreier makes each $\ast $-invertible $\ast $-ideal of $D$ principal
and consequently all members of $\mathcal{I}_{\ast }(D)$ principal.
\end{proof}

This brings us to the last item on the \textquotedblleft
agenda\textquotedblright . In 1998, Professor Halter-Koch wrote a book, \cite%
{H-K 1998} and restated all the then known conepts of multiplicative ideal
theory for monoids, in terms of ideal systems, except for one, he did not
include a Schreier monoid nor a pre-Schreier monoid. Provided below is one
of the missing definitions.

\begin{definition}
\label{Definition WX}A conic, cancellative divisibility monoid $<M,\bullet
,1,\leq >$is a pre-Schreier or a Riesz monoid if every element of $M$ is
primal.
\end{definition}

To end it all let's note, as Professor Halter-Koch would have, that an
integral domain $D$ all nonzero elements of whose multiplicative monoid are
primal is pre-Schreier if $\leq $ is replaced by $|$.

\begin{acknowledgement}
I am grateful to all who I learned from, from Mathematics to painful hard
facts of life. My special thanks go to Professor G.M. Bergman. He often
reminded me of Paul Cohn and has been equally kind. (I have often written
such thanks to him, that disappeared in the final joint versions.)
\end{acknowledgement}

\bigskip

\

\end{document}